\documentclass{amsart}
\usepackage[english]{babel}
\usepackage[T1]{fontenc}
\usepackage{amsmath}
\usepackage{amsthm}
\usepackage{amssymb}
\usepackage{graphicx}
\usepackage{amscd}

\newtheorem{definition}{Definition}[section]
\newtheorem{theorem}[definition]{Theorem}
\newtheorem{lemma}[definition]{Lemma}
\newtheorem{proposition}[definition]{Proposition}

\newtheorem{remark}{Remark}
\newtheorem{example}{Example}

\usepackage{color}

\def\RE {\mathbb{R}}
\def \W  {{d}}
\def \QQ{ {\mathcal{S}}}
\def\a {\alpha}
\def\s {\sigma}
\def\CL { \mathcal{L}}

\def\E {\mathbb{E}}

\def\P {\mathbb{P}}
\def\R {\mathbb{R}}
\def\N {\mathbb{N}}
\def\I {\mathbb{I}}

\def\CF {\mathcal{F}}
\def\CM {\mathcal{M}}
\def\CL {\mathcal{L}}
\def\CB {\mathcal{B}}
\def\CS {\mathcal{S}}

\def\lmq {\lambda_q}
\def\lm {\lambda}

\def \veps {\varepsilon}
\def \vphi {\varphi}
\def \cp {c^+}
\def \cm {c^-}
\def \tcp {\tilde{c}^+}
\def \tcm {\tilde{c}^-}
\def \Mbar {\overline{M}}

\def \Minf {M^{(\alpha)}_\infty}
\def \bX {\overline{X}_0}
\def \bV {\overline{V}_0}

\begin{document}

\title[Speed of convergence in Wasserstein metrics]{Speed of convergence to equilibrium in Wasserstein metrics
for Kac-like kinetic equations.
}

\author{Federico Bassetti      \and
        Eleonora Perversi}

\address{Universit\`a degli Studi di  Pavia, Dipartimento di Matematica, via Ferrata 1,
 27100 Pavia, Italy}

\email{federico.bassetti@unipv.it}
\email{eleonora.perversi@unipv.it}

\date{\today}


\begin{abstract}

This work deals with a class of one-dimensional measure-valued kinetic equations, which constitute extensions of the Kac caricature.
It is known that if the initial datum belongs to the  domain of normal attraction of an $\alpha$-stable law, the solution
of the equation converges weakly to  a suitable scale mixture of
centered $\alpha$-stable laws.
In this paper we present explicit exponential rates for the convergence to equilibrium in Kantorovich-Wasserstein distances
of order $p>\alpha$, under the natural assumption that the distance
between the initial datum and the limit distribution is finite.
For $\alpha=2$ this assumption reduces to the finiteness of the absolute moment of order $p$ of the initial datum.
On the contrary, when $\alpha<2$, the situation is more problematic due to the fact that both the limit distribution
and the initial datum have infinite absolute moment of any order $p >\alpha$.
For this case, we provide sufficient conditions for the finiteness of the Kantorovich-Wasserstein distance.
\end{abstract}

\maketitle

\section{Introduction}
This paper is concerned with the study
of the  speed of convergence to equilibrium 
$-$ with respect to Wasserstein distances $-$ of the
solution of the one--dimensional kinetic equation
\begin{equation}
  \label{eq.1}
  \left\{ \begin{aligned}
      & \partial_t \mu_t + \mu_t =
      {Q}^+(\mu_t,\mu_t )
       \\
      & \mu_0=\bar \mu_0.\\
    \end{aligned}
  \right.
\end{equation}
 The solution $\mu_t =\mu_t(\cdot)$ is a time-dependent
probability measure on $\CB(\R)$, the Borel $\sigma$-field of $\R$.
Following
 \cite{BassLadMatth10,CeGaBo} we assume
that $Q^+$ is a suitable \textit{smoothing transformation}. More
precisely, the probability measure $Q^+(\mu ,\mu )$ is characterized
by
 \begin{equation}\label{eq.2}
  \int_{\R} g(v) Q^+(\mu ,\mu)(dv)=\E \Big [\int_\RE \int_\RE
g\Big(v_1 L+ v_2 R \big)   \mu(dv_1) \mu(dv_2) \Big ],
 \end{equation}
for all bounded and continuous test functions $g \in
C_b(\RE)$,
where $(L,R)$  is a random vector of $\RE^2$ defined on a
probability space $(\Omega,\CF,\P)$ and $\E$ denotes the expectation
with respect to $\P$.

For  suitable choices of $(L,R)$,
 equation \eqref{eq.1}-\eqref{eq.2}
reduces to well-known simplified models for a spatially homogeneous
gas, in which particles move only in one spatial direction. The
basic assumption is that particles change their velocities only
because of binary collisions. When two particles collide, then their
velocities change from $v$ and $w$, respectively, to
 \begin{equation*}
 v' = L_1 v + R_1w  \qquad w' = R_2v + L_2w
 \end{equation*}
where $(L_1,R_1)$ and $(L_2,R_2)$ are two identically distributed random vectors
 with the same law of $(L,R)$.
A  fundamental hypothesis on  $(L,R)$ in this kind of equation is that there
exists an $\alpha$ in $(0,2]$ such that
\begin{align}
  \label{eq.3 1}
  \E\big[ |L|^\alpha+|R|^\alpha\big] = 1 .
\end{align}

The first model of  type \eqref{eq.1}-\eqref{eq.2} has
been introduced by Kac \cite{Kac}, with  collisional parameters
$L=\sin \tilde \theta$ and $R=\cos  \tilde \theta$
{for a random angle $ \tilde \theta$ uniformly distributed  on
$[0,2\pi)$}. 
The {\it inelastic Kac equation}, introduced in \cite{PulvirentiToscani} to describe gases with
inelastically colliding molecules,
corresponds to \eqref{eq.1}-\eqref{eq.2} with 
$L=|\sin
\tilde \theta|^d\sin \tilde\theta$ and $R=|\cos  \tilde
\theta|^d\cos  \tilde\theta$,  where $d>0$ is the parameter of
inelasticity. In this case,  \eqref{eq.3 1} holds with $\alpha=2/(d+1)$.

A less standard application of equations  of type
 \eqref{eq.1}-\eqref{eq.2} is concerned with the construction of   kinetic models
for conservative economies. These models consider the evolution of wealth
distribution in a market of agents which interact through binary
trades, see for example  \cite{BaLaTo,BCC,MatthesToscani,NPT}.

Finally, we mention that, using  results in \cite{CeGaBoBis},  it can be shown
that  the  isotropic solutions of the multidimensional  inelastic homogeneous Boltzmann equation \cite{BoCe}
are functions of one-dimensional $\mu_t$ that are solutions
of equation  \eqref{eq.1}-\eqref{eq.2} for a suitable choice of $(L,R)$
and $\bar \mu_0$.

Recently, the {\it generalized Kac-equation}  \eqref{eq.1}-\eqref{eq.2}
has been extensively studied in many aspects. In particular, the asymptotic behavior of the solutions
of  \eqref{eq.1}-\eqref{eq.2}  has been satisfactory treated in \cite{BaLa,BassLadMatth10,CeGaBo},
while the problem of propagation of smoothness has been addressed in
\cite{MatthesToscani,MaTo} when $\alpha=1$ or $\a=2$.

In \cite{BassLadMatth10} it is proved that, {\it if $L$ and $R$ are positive random variables such that \eqref{eq.3 1} holds true
for $\alpha \in (0,1) \cup (1,2]$, $\E[L^p+R^p]<1$ for some $p >\alpha$  and  $\bar \mu_0$ belongs to the  domain of normal attraction
of an $\alpha$-stable law {\rm(}$\bar \mu_0$ being centered if $\alpha > 1${\rm)}, then
the solution $\mu_t$  converges weakly to a probability measure $\mu_\infty$, that is a mixture of centered $\alpha$-stable distributions.} 
Some extra conditions are needed for the case $\alpha=1$, but the result is essentially
of the same type. For a precise statement of these results, see Theorems \ref{thm2} and \ref{thm1} in Section \ref{sec:convergenceSteady}.
As for the limit distribution, it is easy to see that $\mu_\infty$ is a steady state, that is a fixed point of the smoothing transformation $Q^+$. Moreover,
it has been proved that also the mixing distribution is a fixed point of another smoothing transformation. For more information on fixed points of smoothing transformations
see \cite{DurrettLiggett1983}. See also the very recent paper
\cite{alsmeyerbiggins} and the references therein.

In addition to the problem of finding sufficient (and eventually necessary, see e.g. \cite{GabettaRegazzini2012}) conditions
for the relaxation to the steady state, an important problem is to
determine explicit rates of convergence to the equilibrium with respect to suitable probability metrics.

In the case of the Kac equation, that has the Gaussian distribution as steady state,
rates of convergence with respect to Kolmogorov's uniform metric,  weighted $\chi$-metrics of order $p \geq 2$,
Wasserstein metrics of order 1 and 2  and total variation distance have been proved.
See \cite{dolera,dolera2,GabettaRegazziniWM}.
As for the inelastic Kac equation, in \cite{BaLaRe}
rates of convergence to
equilibrium with respect to
Kolmogorov's uniform metric and  $\chi$-weighted metrics have been derived.
For the solutions of the general model  \eqref{eq.1}-\eqref{eq.2} less is known.
Some results for the Wasserstein distances, of order $p\leq 2$
have been proved in \cite{BaLa,BassLadMatth10}.

The aim of this article is to prove  new exponential bounds for
the speed of approach to equilibrium for the solution of \eqref{eq.1}-\eqref{eq.2}
with respect to Wasserstein
metrics of any order.

Our main results from Theorems \ref{alfa<1}, \ref{alfa1} and \ref{alfa2} can be summarized as follows:

{\it Assume that $L$ and $R$ are positive random variables such that $\P\{L>0\}+\P\{R>0\}>1$,
\eqref{eq.3 1} holds with $\alpha \in (0,1) \cup (1,2]$ and $\E[L^p+R^p]<1$
for some $p>\alpha$. If $\bar{\mu}_0$ belongs to the  domain of normal attraction
of an $\alpha$-stable law {\rm(}$\bar \mu_0$ being centered if $\alpha > 1${\rm)}
and the Wasserstein distance $d_p(\bar \mu_0,\mu_\infty)$ is finite, then
\[
 d_p(\mu_t,\mu_\infty) \leq C_{\bar \mu_0,p} e^{-K_{\alpha,p} t}
\]
for suitable positive constants $C_{\bar \mu_0,p}$ and $K_{\alpha,p}$. }

A similar result holds for $\alpha=1$, see Theorem \ref{alfa1}. The constant $K_{\alpha,p}$, that
will be explicitly computed for $\a<2$, depends only on the law of $(L,R)$, while
$C_{\bar \mu_0,p}$ depends also  on $\bar \mu_0$ and is finite if $d_p(\bar \mu_0,\mu_\infty)<+\infty$.
It is worth noticing that, if $\alpha<2$, the assumption $\W_{p}(\bar \mu_0,\mu_\infty)<+\infty$ is a
non-trivial requirement, since, with the exception of some degenerate case,
one has that $\int_\R |x|^p\bar{\mu}_0(dx)=+\infty$ and $\int_\R |x|^p{\mu}_\infty(dx)=+\infty$
for every $p>\alpha$. For this reason, sufficient conditions for the finiteness of $\W_{p}(\bar \mu_0,\mu_\infty)$
will be presented.

The rest of the paper is organized as follows:
Section \ref{S:preliminaryresults} contains a brief summary of
some known results on the relaxation to equilibrium for the solution of
equation \eqref{eq.1}-\eqref{eq.2}. Section \ref{sct.results}
contains the main results of the paper. More specifically, Subsection \ref{sec:mainresults} presents the exponential bound
for the Wasserstein distance $d_p(\mu_t,\mu_\infty)$
in the case $\alpha <2$. Subsection  \ref{finitenessdp} contains some sufficient condition
for  $\W_{p}(\bar \mu_0,\mu_\infty)<+\infty$ when $\alpha<2$.
Finally, Subsection \ref{sec:mainresults2} treats
the case $\a=2$. The proofs are collected in Sections \ref{Sec_proof1}-\ref{Sec_proofUltima}.

\section{Preliminary results}\label{S:preliminaryresults}

The following assumption will be needed throughout the paper.
 \vskip 0.3cm
\noindent  {\bf Assumption $(H_0)$:} {\it
 $L$ and $R$ are non-negative random variables such that
 \begin{equation}\label{S(0)}
\P\{L>0\}+\P\{R>0\}>1,
\end{equation}
moreover there exist $\alpha$ in $(0,2]$ and $p>0$ satisfying
 \begin{equation}
  \label{eq.3bis}
  \E\big[ L^\alpha+R^\alpha\big] = 1
\end{equation}
  and
\begin{equation}
  \label{eq.3_4}
  \E\big[ L^p+R^p \big] < 1.
\end{equation}
  }
\vskip 0.3cm

For later reference, introduce the convex function $\QQ:[0,\infty)\to[-1,\infty]$ by
\begin{equation*}
  \QQ(s)=\E[L^s+R^s]-1,
\end{equation*}
with the convention that $0^0=0$.
Clearly, under $(H_0)$, $\QQ(\alpha)=0$ and $\QQ(p)<0$. In addition, one has that
\begin{equation}\label{H1}
\P\{(L,R)\in\{0,1\}^2\}<1\quad\text{and}\quad\CS(0)>0.
\end{equation}

\subsection{Probabilistic representation of the solution}\label{probrep}
In this paper we shall use the Fourier formulation of \eqref{eq.1}.
We say that $\mu_t$ is a (weak) solution of
\eqref{eq.1}, with initial condition $\bar \mu_0$, if its
Fourier-Stieltjes transform
 $\hat{\mu}_t(\xi)=\int_\R e^{i\xi v}
\mu_t(dv)$ obeys to the equation
\begin{equation}
  \label{eq.boltzivp}
  \left\{ \begin{aligned}
      & \partial_t\hat{\mu}_t (\xi)+\hat{\mu}_t(\xi) =
      \widehat{Q}^+[\hat{\mu}_t,\hat{\mu}_t](\xi)
       \qquad(t>0, \xi \in  \RE)\\
      & \hat{\mu}_0(\xi):=\int_\R e^{i\xi v} \bar \mu_0(dv)\\
    \end{aligned}
  \right.
\end{equation}
where
\begin{align}
  \label{eq.collop}
  \widehat{Q}^+ [ f,g] (\xi)
  := \E[f(L\xi)g(R\xi)]
\end{align}
for any couple of characteristic functions $(f,g)$.

As in the case of the Kac equation,  it is easy to see that (\ref{eq.boltzivp}) admits a unique solution $\hat{\mu}_t$ (in the class of
the Fourier-Stieltjes transforms) which can be written as  a Wild
series \cite{Wild1951}
\begin{align}
\label{Wild1} \hat{\mu}_t(\xi)=\sum_{n \geq 0} e^{-t}(1-e^{-t})^{n} q_n(\xi),
\end{align}
where $q_0(\xi):=\hat{\mu}_0(\xi)$ and, for $n \geq 1$,
\begin{align}
\label{Wild2}
q_n(\xi):=\frac{1}{n} \sum_{j= 0}^{n-1} \widehat Q^+(q_j,q_{n-1-j})(\xi).
\end{align}

In \cite{BassLadMatth10} it has been shown that
the solution of  \eqref{eq.1} is related
to a suitable stochastic process.
More precisely, the unique solution $\mu_t$ of \eqref{eq.1} with initial datum $\bar{\mu}_0$ is the law of
the weighted random sum
\begin{equation*}
V_t:=\sum_{j=1}^{N_t}\beta_{j,N_t}X_j,
\end{equation*}
with the following elements defined on a sufficiently large probability space $(\Omega, \CF, \P)$:
\begin{itemize}
\item a sequence $(X_j)_{j\geq1}$ of i.i.d. random variables with distribution $\bar\mu_0$;
\item a stochastic process $(N_t)_{t\geq0}$ which takes values in $\N$ and with
        \[
        \P\{N_t=n\}=e^{-t}(1-e^{-t})^{n-1}
        \]
         for every $n\geq1$ and $t\geq0$;
\item  a random array of weights $(\beta_{j,n}:\;j=1,\dots,n)_{n\geq1}$ recursively defined by:
        \begin{equation*}
        \left\{
        \begin{aligned}
        & \beta_{1,1}:=1\\
        & (\beta_{1,2},\beta_{2,2}):=(L_1,R_1)\\
        & (\beta_{1,n+1},\dots,\beta_{n+1,n+1}) \\
&\qquad \qquad :=(\beta_{1,n},\dots,\beta_{I_n-1,n},L_n\beta_{I_n,n},R_n\beta_{I_n,n},\beta_{I_n+1,n},\dots,\beta_{n,n}). \\
        \end{aligned}
        \right.
        \end{equation*}
where $(L_n,R_n)_{n\geq1}$ is a sequence of independent and identically distributed (i.i.d., for short) random vectors with the same distribution of $(L,R)$, and
 $(I_n)_{n\geq1}$ is a sequence of independent random variables such that $I_n$ is uniformly distributed on $\{1,\dots,n\}$ for every $n\geq1$;
\item $(X_j)_{j\geq1}$, $(N_t)_{t\geq0}$, $(L_n,R_n)_{n\geq1}$, $(I_n)_{n\geq1}$ are stochastically independent.
\end{itemize}
As a matter of fact, it is possible to prove that for every $n\geq1$, $\hat{q}_{n-1}$ $-$ defined in \eqref{Wild2} $-$ is the characteristic
function of the random variable
\begin{equation}\label{Wn}
W_n:=\sum_{j=1}^n\beta_{j,n}X_j.
\end{equation}
See the proof of Proposition 1 in \cite{BassLadMatth10}. Since  $V_t=W_{N_t}$, from \eqref{Wild1} it follows   that $\mu_t$ is the law of $V_t$.

\subsection{Martingale of weights and fixed point equations for distributions}
It is easy to prove that, under $(H_0)$, $\sum_{j=1}^n\beta_{j,n}^\alpha$
is a (positive) martingale and hence it converges a.s. (as $n \to +\infty$) to a random
variable $M_{\infty}^{(\alpha)}$. Moreover, $M_{\infty}^{(\alpha)}$  satisfies the fixed point equation
for distributions
\begin{equation}
  \label{eq.3}
M_{\infty}^{(\alpha)}\stackrel{d}=L^\alpha M_{\infty,1}^{(\alpha)} + R^\alpha M_{\infty,2}^{(\alpha)}.
\end{equation}
In \eqref{eq.3}, $M_{\infty,1}^{(\alpha)}$, $M_{\infty,2}^{(\alpha)}$ and $(L,R)$
are stochastically independent,  $M_{\infty,1}^{(\alpha)}$ and  $M_{\infty,2}^{(\alpha)}$  have the same law of
 $M_{\infty}^{(\alpha)}$,  and $Z_1\stackrel{d}=Z_2$ means that
the random variables $Z_1$ and $Z_2$ have the same distribution.
For a proof of these facts see Proposition 2 in \cite{BassLadMatth10}.

Note that equation \eqref{eq.3} can be written in terms of the characteristic function $\hat \nu_\a(\xi)=\E[\exp\{i\xi M_{\infty}^{(\alpha)}\}]$ 
as
  \begin{equation}\label{eq-M}
\hat \nu_\alpha  (\xi)=\E[\hat \nu_\alpha  (L^\alpha \xi) \hat \nu_\alpha  (R^\alpha \xi)  ] \qquad (\xi \in \RE).
\end{equation}

In the next proposition we collect some useful properties of the solution of equations \eqref{eq.3}-\eqref{eq-M}.

\begin{proposition}[\cite{alsmeyerbiggins,DurrettLiggett1983,Liu1998}]\label{Prop0-1}
  Let $(H_0)$ be in force with $\alpha<p$.
Then, there is a unique probability distribution   $\nu_\alpha$ on $\CB(\RE^+)$ with $\int_{\RE^+}  v \nu_\alpha(dv)=1$
and Fourier-Stiletjes transform $\hat \nu_\alpha(\xi)=\int_{\RE}e^{i\xi v}\nu_\alpha(dv)$ satisfying equation \eqref{eq-M}.
Moreover,
   \begin{itemize}
  \item[(i)] If  $L^\alpha+R^\alpha =1$ almost surely, %
  then  $\nu_\alpha(\cdot)=\delta_1(\cdot)$;
  \item[(ii)] If $\P\{ L^\alpha+R^\alpha=1\}<1$,
       then $\nu_\alpha$ is non-degenerate and, for any $q>\alpha$,
    $\int_{\RE^+} v^{\frac{q}{\alpha}} \nu_\alpha(dv)<+\infty$ if and only if $\CS(q)<0$.
   \end{itemize}
\end{proposition}


\subsection{Stable laws}
Recall that
a probability distribution  $ g_\a$ is said to be
{\em a centered stable law} of exponent $\a$ (with $0 < \a \leq 2$) and real parameters $(\lm,\beta)$, $\lm>0$ and $|\beta|\leq1$, if its Fourier-Stieltjes transform $\hat g_\alpha(\xi)=\int_\R e^{i\xi v} g_\alpha(dv)$  has the form
\begin{equation}
  \label{chaSta}
  \hat g_\a(\xi)=
  \left \{
    \begin{array}{ll}
      \exp\{ -\lm |\xi|^\a
      (1-i \beta \tan(\pi\alpha/2)\operatorname{sign}\xi)   \} &  \text{if $\a \in (0,1) \cup (1,2)$}  \\
      \exp\{ -\lm |\xi|
      (1+2i\beta/\pi \log|\xi| \operatorname{sign}\xi)   \} &  \text{if $\a=1$} \\
      \exp\{ - \lm |\xi|^2 \}\nobreakspace& \text{if $\a=2$.}
    \end{array}
  \right .
\end{equation}

By definition, a probability measure $\bar \mu_0$
belongs to the {\em   domain of normal attraction} of
a stable law of exponent $\a$ if for any sequence of i.i.d. real-valued
random variables $(X_n)_{n \geq 1}$
with common distribution  $\bar \mu_0$,
there exists a sequence of real numbers $(c_n)_{n \geq 1}$
such that the law of
\(
{n^{-1/\a}} \sum_{i=1}^n X_i -c_n
\)
converges weakly to a stable law of exponent $\a$.

It is well-known  that, provided $\a\not=2$,
a probability measure $\bar \mu_0$ belongs to the { domain of normal attraction} of an $\a$-stable law
if and only if its distribution function $F_0(x):=\bar \mu_0\Big((-\infty,x]\Big)$ satisfies
\begin{equation}
  \label{stabledomain}
    \lim_{x \to +\infty} x^\a (1-F_0(x)) =c_0^+<+\infty, \quad
    \lim_{x \to -\infty} |x|^\a F_0(x) =c_0^-<+\infty. 
\end{equation}
Typically, one also requires that $c_0^++c_0^->0$ in order to exclude
convergence to the probability measure concentrated in $0$,
but here we shall include the situation $c_0^+=c_0^-=0$ as a special case.
The parameters $\lm$ and $\beta$ of the associated stable law
in \eqref{chaSta}
are related to $c_0^+$ and $c_0^-$ by
\begin{equation}
  \label{constant}
  \lm =  \frac{(c_0^{+}+c_0^{-})\pi}{2\Gamma(\a)\sin(\pi\a/2)},
  \qquad \beta = \frac{c_0^{+}-c_0^{-}  }{c_0^{+}+c_0^{-}},
\end{equation}
with the convention that $\beta=0$ if $c_0^++c_0^-=0$.
In contrast, if $\a=2$,
$F_0$ belongs to the   domain of normal attraction of a Gaussian law
if and only if it has finite variance $\s^2$. The parameter $\lm$ of the associated Gaussian law in \eqref{chaSta} is given by $\lm=\frac{\sigma^2}{2}$.
%
See for example Chapter 17 of \cite{fristedgray} and Chapter 2 of \cite{Ibragimov}.

\subsection{Convergence to Steady states}\label{sec:convergenceSteady}

We are ready to state the results concerning the convergence
of $\mu_t$ to a steady state, that is a probability measure $\mu_\infty$ such that
\[
\mu_\infty=Q^+(\mu_\infty,\mu_\infty).
\]

\begin{theorem}[\cite{BassLadMatth10}]\label{thm2}
  Assume that $(H_0)$ holds true with $\a \not=1$
  and that  $F_0$ satisfies \eqref{stabledomain}. In addition,  assume that
  $\int_\R v \bar \mu_0(dv)=0$ if $\alpha>1$.
  If $p<\alpha$, then $\mu_t$ converges weakly to the degenerate probability measure $\delta_0$,
  while, if $p>\alpha$, then $\mu_t$ converges weakly to a steady state
  $\mu_\infty$ with Fourier-Stieltjes transform
  \begin{equation}
    \label{characteristic}
    \int_\R e^{i \xi v} \mu_\infty(dv)=\int_{[0,+\infty)}
e^{-\lm m |\xi|^\alpha[1- i\beta\tan(\alpha\frac{\pi}{2})\operatorname{sign}\xi]} \nu_\alpha(dm)
    \qquad (\xi \in \RE),
  \end{equation}
  where
 $\nu_\alpha$ is the same as in {\rm Proposition  \ref{Prop0-1}} and the parameters $\lm$ and $\beta$ are defined in \eqref{constant}
  for $\a<2$ and $(\lm,\beta)=(\s^2/2,0)$ for $\a=2$.
  \end{theorem}

We conclude this section by considering the case in which $\a=1$. We state a slight variant
of Theorem 4 in \cite{BassLadMatth10}.

\begin{theorem}\label{thm1}
  Assume that $(H_0)$ holds with $\a=1$.
  Suppose that $F_0$ satisfies
\begin{equation}\label{NDA-1sym}
\lim_{x\to-\infty}|x|F_0(x)=\lim_{x\to+\infty}x\Big[1-F_0(x)\Big]=c_0\in [0,+\infty)
\end{equation}
and suppose, in addition, that
\begin{equation}\label{gamma}
\lim_{R\to+\infty}\int_{(-R,R)}x dF_0(x)=\gamma_0
\end{equation}
with $-\infty<\gamma_0<+\infty$.
 If $p<1$, then $\mu_t$ converges weakly to the degenerate probability measure $\delta_0$,
 while, if $p>1$, then $\mu_t$ converges weakly, as $t\to +\infty$, to a steady state $\mu_\infty$ with Fourier-Stieltjes transform
\begin{equation}\label{characteristic2}
\int_\RE e^{i \xi v} {\mu}_\infty(dv)=\int_{\RE^+} e^{ m(i  \gamma_0 \xi- c_0\pi|\xi|)} \nu_1(dm)
\end{equation}
where $\nu_1$ is the same as in {\rm Proposition  \ref{Prop0-1}}.
\end{theorem}

This theorem can be proved in a very similar way of Theorem 1 of \cite{BassLadMatth10}, for the sake of completeness 
a sketch of the proof is given in Appendix \ref{AppendixB}.


\begin{remark}
It is worth noticing that the steady states $\mu_\infty$ described in Theorems \ref{thm2}-\ref{thm1} 
are the unique possible fixed points of $Q^+$. See Theorems 2.1 and 2.2 in \cite{alsmeyerbiggins}. 
Necessary  conditions for the convergence of $\mu_t$ to a steady state $\mu_\infty$ 
are investigated in \cite{PerversiRegazzini}.
\end{remark}

\section{Rates of convergence in Wasserstein distances}\label{sct.results}
The {\it minimal $L_p$-metric} $-$ or {\it Kantorovich-Wasserstein distance} of order $p$
$-$
($p>0$)
between two probability measures $\mu_1$ and $\mu_2$ on $\CB(\R)$ is defined by
\begin{equation}\label{eq.wasserstein}
   \W_p(\mu_1,\mu_2):=\inf_{ m \in \CM(\mu_1,\mu_2)} \Big (\int_{\RE^2} |x-y |^p m(dxdy) \Big) ^{1\wedge1/p},
\end{equation}
where $\CM(\mu_1,\mu_2)$ is the class of all the probability measures on $\CB(\RE^2)$
with marginals $\mu_1$ and $\mu_2$, that is the probability measures $m$
such that $m(\cdot \times \RE)=\mu_1(\cdot)$ and $m(\RE \times \cdot)=\mu_2(\cdot)$.
In general, the infimum in \eqref{eq.wasserstein} may be infinite;
a sufficient (but not necessary) condition for having finite distance between $\mu_1$ and $\mu_2$ is
that both $\int_\R |v|^p \mu_1(dv)<+\infty$ and $\int_\R |v|^p \mu_2(dv)<+\infty$. An important property of the Kantorovich-Wasserstein
distance is its close connection with weak convergence of probability measures;
namely, if $(\nu_t)_{t\geq0}$ is a family
of probability measures such that $\int_\R |v|^p \nu_t(dv)<+\infty$ for every $t\geq0$ and
 $\nu_\infty$ is a probability measure such that $\int_\R |v|^p \nu_\infty(dv)<+\infty$,
 then $\W_p(\nu_t,\nu_\infty) \to 0$, as
 $t \to +\infty$, if and only if
$\nu_t $ converges weakly to $\nu_\infty$ and
 \[
\int_\R |x|^p \nu_t(dx) \to \int_\R |x|^p \nu_\infty(dx)\quad\text{for $t\to\infty$}.
\]
See, e.g., Lemma 8.4.35 in \cite{RachevRuschendorf}. Recall also that $\W_p(\nu_t,\nu_\infty) \to 0$, as $t\to+\infty$,
yields the weak convergence of $\nu_t$ to $\nu_\infty$, even if
$\int_\R |v|^p \nu_t(dv)=+\infty$ for every $t\geq0$.

In the rest of the section we deal with the problem of providing an upper bound for
$\W_p(\mu_t,\mu_\infty)$ when $\mu_t$
is the solution of \eqref{eq.1} with initial condition $\bar \mu_0$
  and $\mu_\infty$ is the corresponding steady state.

When $\a\not=1,2$, taking advantage of a probabilistic representation of the solution
recalled in Section \ref{probrep},
it is relatively easy to get an upper bound for $\W_p(\mu_t,\mu_\infty)$ whenever $p \leq 2$. The reason of the
restriction to $p\leq 2$ is that in proving such kind of
estimates a key point is the employment of the von Bahr - Esseen inequality for sums of independent random variables  -- see \eqref{vonBahrEsseen} --,
which holds only if $p\leq 2$. In order to enunciate these rates of convergence we recall that the so-called \textit{spectral function},
introduced in \cite{CeGaBo}, is the function  $\vphi\colon (0,+\infty)\to \overline{\R}:=\R \cup \{-\infty,+\infty\}$ defined
by
\begin{equation}\label{defPhi}
\vphi(q):=\frac{\CS(q)}{q}.
\end{equation}

\begin{theorem}[\cite{BassLadMatth10}]
  \label{PropW-2}
Let the same assumptions of {\rm Theorem \ref{thm2}} be in force
for some $p$ with $1< \a<p\leq2$
  or $\a<p \leq 1$. If $\W_{p}(\bar \mu_0,\mu_\infty) <+\infty$, then
  \begin{equation*}
    \W_{p}(\mu_t,\mu_\infty) \leq A^{\frac{1}{p}\wedge 1} \W_{p}(\bar \mu_0,\mu_\infty) e^{-t|\varphi (p)| (p \wedge 1)} ,
  \end{equation*}
  with $A=1$ if $p\leq1$, or $A=2$ otherwise.
\end{theorem}

\begin{remark} It is worth noticing
that, if $\alpha<2$ and $\cp_0+\cm_0>0$, the assumption $\W_{p}(\bar \mu_0,\mu_\infty)<+\infty$ is a
non-trivial requirement, since $\int_\R |x|^p\bar{\mu}_0(dx)=+\infty$ and $\int_\R |x|^p{\mu}_\infty(dx)=+\infty$
for every $p>\alpha$.
In {\rm Section \ref{finitenessdp}} we will give sufficient conditions for the finiteness of $\W_{p}(\bar \mu_0,\mu_\infty)$.
\end{remark}

Theorem \ref{PropW-2} does not cover the cases $\alpha=1$ and $\alpha=2$ and the cases $\a\in(0,1)$ and $p>1$ or $\a\in(1,2)$ and $p>2$. In the next sections we will plug this gap.

\subsection{Statement of the main results for $\alpha <2$}\label{sec:mainresults}
In this section we will enunciate two results which provide (exponential) rates of
convergence to equilibrium for the solution of \eqref{eq.1} with respect to the Wasserstein distances of any order.
The proofs of these statements will be established by using the probabilistic representation of the solution of \eqref{eq.1} and employing an
inductive argument inspired by a technique developed in \cite{FillJans01}. This inductive argument makes use of rates of convergence to equilibrium with
respect to Wasserstein distances of order $p\leq2$; thus, it is crucial to have estimates for $d_p(\mu_t,\mu_\infty)$ when $p\leq2$.
Theorem \ref{PropW-2} fulfills our need if $\alpha\neq1$, while, when $\alpha=1$, we have to prove an estimate that will make us able to proceed with the next inductive argument.
This key step is provided by the following theorem.

\begin{theorem}\label{PropW-3}
  Assume that $(H_0)$ holds true with $\a=1$ and $1<p\leq2$, and that
  $\bar \mu_0$ satisfies the assumptions of {\rm Theorem \ref{thm1}}.
  If $\W_{p}(\bar \mu_0,\mu_\infty) <+\infty$, then
  \begin{equation}
    \label{main_bound}
    \W_{p}(\mu_t,\mu_\infty) \leq C_p e^{-t|\vphi(p)|} ,
  \end{equation}
for a suitable constant $C_p=C_p(\bar \mu_0)<+\infty$.
\end{theorem}

Note that if $\int_\R |v|\bar \mu_0(dv)<+\infty$, then $c_0=0$, $\gamma_0=\int_\R v\bar \mu_0(dv)$ and
$\mu_\infty(\cdot)=\nu_1(\cdot/\gamma_0)$. By Proposition \ref{Prop0-1} (ii), since $\CS(p)<0$, we know that $\int_{\R^+}v^p\nu_1(dv)<+\infty$ and hence $\int_{\R}|v|^p\mu_\infty(dv)<+\infty$. Thus, $d_p(\bar{\mu}_0,\mu_\infty)<+\infty$ if and only if $\int_\R |v|^p \bar{ \mu}_0(dv)<+\infty$
and Theorem \ref{PropW-3} reduces to Theorem 5 of \cite{BassLadMatth10}. Analogously, if $\bar \mu_0$ is symmetric
and satisfies \eqref{NDA-1sym}, then the previous theorem reduces to Theorem 2.4 in \cite{BaLa}.

In order to introduce the generalizations of Theorems \ref{PropW-2} and \ref{PropW-3} to Kantorovich-Wasserstein metrics
of higher order, we define, for $i=1,2$ and every $q\geq i$,
\[
K_i(q):= \max\{\vphi(i),\vphi(q)\}.
\]
We are now in the position to enounciate the aforementioned exponential rates of convergence, which are divided into two different theorems according to the value of $\alpha$.

\begin{theorem}[$0<\alpha<1$]\label{alfa<1}
 Assume that $(H_0)$ holds true with $0 < \a <1$ and  $p> 1$. Assume also that
  $\bar \mu_0$ satisfies the hypotheses of {\rm Theorem \ref{thm2}} and that $d_p(\bar{\mu}_0,\mu_\infty)<+\infty$. Then there exists a constant
$C_p=C_p(\bar \mu_0)<+\infty$ such that
\begin{equation}\label{ratea<1}
d_p(\mu_t,\mu_\infty)\leq \left\{
													\begin{array}{ll}
													C_p e^{-t|K_1(p)|} & \text{if $\vphi(p)\neq\vphi(1)$}\\
													C_p t e^{-t|K_1(p)|} & \text{if $\vphi(p)=\vphi(1)$}
													\end{array}
													\right.
\end{equation}
for every $t \geq 0$. 
\end{theorem}

\begin{theorem}[$1 \leq\alpha<2$]\label{alfa1}
 Assume that $(H_0)$ holds true with $1 \leq \a <2$ and $p>2$. If $\alpha=1$ suppose that
  $\bar \mu_0$ satisfies the hypotheses of {\rm Theorem \ref{thm1}},
while if $1<\a<2$ assume that  $\bar \mu_0$ satisfies the hypotheses of
{\rm Theorem \ref{thm2}}.  Assume also that $d_p(\bar{\mu}_0,\mu_\infty)<+\infty$. Then there exists a  constant
$C_p=C_p(\bar \mu_0)<+\infty$ such that
\begin{equation}\label{rate1}
d_p(\mu_t,\mu_\infty)\leq \left\{
													\begin{array}{ll}
													C_p e^{-t|K_2(p)|} & \text{if $\vphi(p)\neq\vphi(2)$}\\
													C_p te^{-t|K_2(p)|} & \text{if $\vphi(p)=\vphi(2)$}
													\end{array}
													\right.
\end{equation}
for every $t \geq 0$. 
\end{theorem}

\begin{example}
Let us consider the case in which $L=1-R=U$ where $U$ is a random variable uniformly distibuted on $(0,1)$. In this special case $\CS(s)=\frac{1-s}{1+s}$ and $\vphi(s)=\frac{1-s}{s(1+s)}$. Since $0=\CS(1)>\CS(p)$ for every $p>1$, Theorem \ref{thm1} can be applied. In particular, using also Proposition \ref{Prop0-1} (i), we have that $\nu_1=\delta_1$ and $\mu_\infty$ is a Cauchy distribution of scale parameter $\pi c_0$ and position parameter $\gamma_0$. Noticing that $\vphi(2)=\vphi(3)=-1/6$, Lemma \ref{unicominimo} in Section \ref{section5} entails that Theorem \ref{alfa1} holds with
\[
K_2(p)=\left\{
				\begin{array}{ll}
				-1/6 & \text{if $2\leq p\leq 3$}\\
				(1-p)/(p+p^2) & \text{if $p>3$.}
				\end{array}
				\right.
\]
\end{example}

\begin{example}
Another interesting example is the case of the inelastic Kac equation \cite{PulvirentiToscani}. The inelastic Kac equation can be reduced to a special case of equation \eqref{eq.1}-\eqref{eq.2} with 
$L=|\cos(\tilde \theta)|^{1+d}$ and $R=|\sin(\tilde \theta)|^{1+d}$, $\tilde \theta$ being a random variable uniformly distributed on $(0,2\pi)$ and $d>0$.
In this case
\[
\begin{split}
 \CS(s)&=\frac{1}{2\pi}\int_{(0,2\pi)}(|\sin(\theta)|^{(1+d)s}+|\cos(\theta)|^{(1+d)s})d\theta-1\\
&= 
\frac{1}{\pi}\int_{(0,2\pi)}|\sin(\theta)|^{(1+d)s}d\theta-1=\frac{2}{\sqrt{\pi}} \frac{\Gamma(\frac{d+1}{2}s+\frac{1}{2})}{\Gamma(\frac{d+1}{2}s+1)}-1\\
\end{split}
\]
where $\Gamma(x)=\int_0^{+\infty} t^{x-1}e^{-t}dt$. Clearly $\CS(\alpha)=0$ for $\alpha=2/(d+1)$, moreover 
$\CS(p)<0$ for every $p>\alpha$, so that  Theorems \ref{thm2}-\ref{thm1} can be applied. As before,  $\nu_\a=\delta_1$ and $\mu_\infty$
is an $\a$-stable distribution. 
Since $\lim_{s \to +\infty}\CS(s)=-1$, then $\lim_{s \to +\infty}\varphi(s) = 0$  and,  invoking Lemmma  \ref{unicominimo},  
one proves  that $\varphi(s)$ has a unique minimum point in $p_{0}^{(d)}$.
Clearly $p_{0}^{(d)}=p_0^{(1)} 2/(d+1)$ where 
$p_0^{(1)}$ is the unique minimum point of 
\[
 s \to \frac{1}{s}\Big(\frac{2}{\sqrt{\pi}}\frac{\Gamma(s+\frac{1}{2})}{\Gamma(s+1)}-1\Big).
\]
Numerically one sees that $p^{(1)}_0 \approx 2.413$. On the one hand, it is easy to check that if $d\leq1$, i.e. $\a\geq1$, one has $p_{0}^{(d)}>2$. Hence, in this case, there exists a point $p^*_d>2$ such that $-K_2(p)=\vphi(2)$ if $2<p<p^*_d$ and $-K_2(p)=\vphi(p)$ if $p\geq p^*_d$. On the other hand, if $d>1$, i.e. $\a<1$, one has two different situations: (i) $p^{(d)}_0\leq1$ whenever 
$d\geq 2 p_0^{(1)}-1\approx 3.826$, thus $-K_1(p)=\vphi(p)$ for every $p\geq1$; (ii) $p^{(d)}_0>1$ whenever $d< 2 p_0^{(1)}-1\approx 3.826$, thus $-K_1(p)=\vphi(1)$ if $1<p<p^*_d$ and $-K_1(p)=\vphi(p)$ if $p\geq p^*_d$ for a suitable $p^*_d>1$.

\end{example}

\subsection{Asymptotic expansion for the tails of $\mu_\infty$ and sufficient conditions for the finiteness of $d_p(\bar{\mu}_0,\mu_\infty)$ when $\alpha<2$}\label{finitenessdp}
In the theorems of the previous subsection the constants $C_p$ -- which could be explicitly computed in the proofs of Theorems \ref{alfa<1} and \ref{alfa1} --
depend on $d_p(\bar{\mu}_0,\mu_\infty)$ and hence the assumption $d_p(\bar{\mu}_0,\mu_\infty)<+\infty$ is a fundamental requirement for \eqref{ratea<1} and \eqref{rate1} to be meaningful. In some particular cases this assumption reduces to a simpler hypothesis on the finiteness of the absolute $p$-th moment of the initial datum $\bar{\mu}_0$. 
More precisely, as already noted after Theorem \ref{PropW-3}, if $\alpha=1$ and $\int_\R |v|\bar{\mu}_0(dv)<+\infty$,
then 
$d_p(\bar{\mu}_0,\mu_\infty)<+\infty$ if and only if
$\int_\R |v|^p \bar{\mu}_0(dv)<+\infty$. Furthermore, if $\alpha\in(0,1)\cup(1,2)$ and $\cp_0+\cm_0=0$, then $\mu_\infty=\delta_0$,  and therefore $d_p(\mu_t,\mu_\infty)$ $-$ in
Theorems \ref{alfa<1} and \ref{alfa1} $-$ reduces to the absolute moment of  order $p$ of $\mu_t$. In particular,  $d_p(\bar{\mu}_0,\mu_\infty)<+\infty$ holds true if and only if $\int_\R |x|^p\bar{\mu}_0(dx)<+\infty$.
 All the other cases are more problematic. Indeed, as already recalled,
if $\alpha<2$ and $\cp_0+\cm_0>0$, then $\int_\R |x|^p\bar{\mu}_0(dx)=+\infty$ as well $\int_\R |x|^p{\mu}_\infty(dx)=+\infty$ for every $p>\alpha$.

Here we  give a  criterion that provides the finiteness of $d_p(\bar{\mu}_0,\mu_\infty)$ when $p >\alpha$.
The main result of this section is contained in Theorem \ref{teoremaDistFinita} which extends Lemma 1 of \cite{BassLadMatth10}.
 \newline
Let us start by noticing that \eqref{characteristic} can be immediately rewritten in terms of random variables
as follows: under the hypotheses of Proposition \ref{Prop0-1} and Theorem \ref{thm2}, let $\Minf$ be the unique solution of equation \eqref{eq.3}, consider  an $\alpha$-stable random variable $S_\alpha$
of parameters $(\lambda,\beta)$ given by \eqref{constant} and assume that
$\Minf$ and  $S_\alpha$ are stochastically independent. Finally, let $V_\infty$ be a random variable whose probability distribution is $\mu_\infty$. Then, \eqref{characteristic} becomes
\begin{equation}\label{misturaVA}
V_\infty\stackrel{d}{=}S_\alpha \Big(\Minf\Big)^\frac{1}{\alpha}.
\end{equation}
Note that, in the same way, \eqref{characteristic2} becomes
\begin{equation}\label{misturaVA2}
V_\infty\stackrel{d}{=}(S_{1}+\gamma_0)M^{(1)}_\infty=C_{\lm,\gamma_0}M^{(1)}_\infty,
\end{equation}
where $C_{\lm,\gamma_0}$ is a Cauchy random variable of scale parameter $\lm=\pi c_0$ and position parameter $\gamma_0$, and $S_1=C_{\lm,0}$.
In other words, for every $\alpha\in(0,2]$, $V_\infty$ is an $\alpha$-stable random variable randomly
rescaled by $\Big(\Minf\Big)^\frac{1}{\alpha}$.

It is useful to observe that, in order to obtain sufficient conditions for the finiteness of $d_p(\bar{\mu}_0,\mu_\infty)$, when $\alpha=1$
 we can suppose,  without loss of generality, that $\gamma_0=0$. This fact is justified by the next lemma.

\begin{lemma}\label{distfinita1} Let $(H_0)$ hold true with $\alpha=1$ and $p>1$.
Assume that $\bar \mu_0$ satisfies \eqref{NDA-1sym} and \eqref{gamma}, define
$\bar \mu_0^*(\cdot):=\bar \mu_0(\cdot+\gamma_0)$ and let $\mu_\infty^*$ be the corresponding steady state.
Then, $\lim_{R\to+\infty}\int_{(-R,R)}x \bar \mu^*_0(dx)=0$ and
\begin{equation}\label{characteristicmu*}
\int_\RE e^{i\xi v} \mu^*_\infty(dv)=\int_{\RE^+} e^{ -m c_0\pi|\xi|} \nu_1(dm).
\end{equation}
In addition, $ d_p(\bar \mu_0,\mu_\infty)<+\infty$ if and only if $d_p(\bar\mu_0^*,\mu_\infty^*)<+\infty$.
\end{lemma}

Hence, in the rest of this section, we assume that $\gamma_0=0$ whenever $\a=1$.
Under this assumption, \eqref{misturaVA2} reduces to \eqref{misturaVA} and we can write
\begin{equation}\label{Finf-Falfa}
\begin{split}
F_\infty(x)& := \mu_{\infty}\Big((-\infty,x]\Big) =
\P\Big\{S_\alpha \Big(\Minf\Big)^\frac{1}{\alpha}\leq x\Big\}\\
& =\E\left[F_\alpha \left({x}{\Big(\Minf\Big)^{-\frac{1}{\alpha}}}\right)\I_{\{\Minf\neq0\}}+\I_{\{x\geq0\}}\I_{\{\Minf=0\}}\right] \\
\end{split}
\end{equation}
where $F_\alpha$ is the distribution function of $S_\alpha$.
At this stage we can derive a useful asymptotic expansion of $F_\infty$ combining  \eqref{Finf-Falfa} with the well-known asymptotic expansion for the probability
distribution function of a stable law.

\begin{proposition}\label{sviluppoFinf}
Let $0<\a<2$. If $\alpha \not =1$  let the same assumptions of {\rm Theorem \ref{thm2}} hold with $c_0^++c_0^->0$, while if $\alpha=1$
let the same hypotheses of {\rm Theorem \ref{thm1}} be in force with $\gamma_0=0$ and $c_0>0$.
Let $F_\infty$ be the distribution function of the steady state
$\mu_\infty$ described in {\rm Theorem \ref{thm2}}, {\rm Theorem \ref{thm1}} respectively.
Then
\begin{itemize}
\item[(i)] If $\alpha\neq1$, $|\beta|\neq1$ and  $\CS(\alpha(k+\delta))<0$ for some integer  $k \geq 1$ and some $\delta \in (0,1]$,
then $m_i:=\E[(\Minf)^i]<+\infty$  for $i=1,\dots,k$ and
\begin{eqnarray}
\label{codeFinfsx}
&&F_\infty(x)=\frac{\tcm_0}{|x|^{\alpha}}+\frac{\tcm_1}{|x|^{2\alpha}}+\dots+\frac{\tcm_{k-1}}{|x|^{k\alpha}}+O\Big(\frac{1}{|x|^{(k+\delta)\alpha}}\Big)\qquad\text{for $x\rightarrow-\infty$}\\
\label{codeFinfdx}
&&1-F_\infty(x)=\frac{\tcp_0}{x^{\alpha}}+\frac{\tcp_1}{x^{2\alpha}}+\dots+\frac{\tcp_{k-1}}{x^{k\alpha}}+O\Big(\frac{1}{x^{(k+\delta)\alpha}}\Big)\qquad\text{for $x\rightarrow+\infty$}
\end{eqnarray}
where $\tilde{c}^\pm_i:=c^\pm_i m_{i+1}$ for $i=0,\dots,k-1$, with $c^{\pm}_0$ being defined by \eqref{stabledomain} and $(c^\pm_i)_{1\leq i\leq k-1}$ suitable constants {\rm(}see \eqref{Ci} in Appendix A{\rm)}.
If $\alpha\neq1$ and $\beta=-1$ {\rm[}$\beta=1$, resp.{\rm]} and  $\CS(\alpha(k+\delta))<0$, then
\eqref{codeFinfsx}  holds and $1-F_\infty(x)=O\Big(\frac{1}{x^\eta}\Big)$ for $x\rightarrow+\infty$
{\rm[}\eqref{codeFinfdx}  holds and $F_\infty(x)=O\Big(\frac{1}{|x|^\eta}\Big)$ for $x\rightarrow-\infty$, resp.{\rm]}   for every $\eta>0$ such that $\CS(\eta)<0$.\newline
\item[(ii)] If $\alpha=1$ and  $\CS(2k-1+\delta)<0$ for some integer  $k \geq 1$ and $\delta \in (0,2]$,
 then $m_{2i+1}:=\E[(M_{\infty}^{(1)})^{2i+1}]<+\infty$  for $i=0,\dots,k-1$, $F_\infty$ is symmetric and
\begin{equation*}
F_\infty(x)=\sum_{i=0}^{k-1}\frac{\tilde{c}^-_i}{|x|^{2i+1}}+O\left(\frac{1}{|x|^{2k-1+\delta}}\right)\qquad\text{for $x\to-\infty$}.
\end{equation*}
with $\tilde{c}^-_i:=\frac{(-1)^i\lm^{2i+1}m_{2i+1}}{\pi(2i-1)}$ for $i=0,\dots,k-1$.
\end{itemize}
\end{proposition}
For the proof of this proposition the reader is deferred to Appendix A.

It is worth noticing that  $-$  with the exception of few cases, see e.g. \cite{BassettiToscani}  $-$ in general there is no analytical expression of the law of $\Minf$, i.e. $\nu_\a$.
Nevertheless, having an explicit expression of the mixed moment of $(L,R)$,
it is always possible to recursively determine the exact expression of the integer moments of $\nu_\alpha$, i.e.
$m_i:=\E[(\Minf)^i]$. Indeed, $m_1=1$ and, for $i=2,\dots,k$,
 \begin{equation*}
 m_i=\frac{1}{1-\E[L^{\alpha i}+R^{\alpha i}]}\sum_{j=1}^{i-1}\binom{i}{j}\E[L^{\alpha j}R^{\alpha(i-j)}]m_j m_{i-j}.
\end{equation*}
 This recursive formula can be easily obtained using \eqref{eq.3} and Newton binomial formula.
The next theorem provides the announced sufficient conditions on the initial datum $\bar{\mu}_0$ that ensure the
finiteness of $d_p(\bar{\mu}_0,\mu_\infty)$. Essentially, $d_p(\bar{\mu}_0,\mu_\infty)$ is finite whenever
the tails of $F_0$ are close enough to the tails of $F_\infty$. 

\begin{theorem}\label{teoremaDistFinita}
Let $0<\a<2$. If $\alpha \not =1$  let the same assumptions of {\rm Theorem \ref{thm2}} hold with $c_0^++c_0^->0$, while if $\alpha=1$
let the same hypotheses of {\rm Theorem \ref{thm1}} be in force with $\gamma_0=0$ and $c_0>0$. 
Let $p>\alpha$ and set  $k:=\left\lfloor 1+\frac{p-\alpha}{p\alpha}\right\rfloor$.
\begin{itemize}
\item[(i)] Let $|\beta|\neq1$. Assume that $\CS(s)<0$ for some $s>\alpha+(p-\alpha)/p$ and that $F_0$ satisfies
\begin{eqnarray}
\label{condDatoInizSx}
&& \Big|F_0(x)-\sum_{i=0}^{k-1}\frac{\tcm_i}{|x|^{(i+1)\alpha}}\Big|\leq \frac{\zeta(|x|)}{|x|^{(1+\frac{p-\alpha}{p\alpha})\alpha}}\quad\text{for $x\rightarrow-\infty$}\\
\label{condDatoInizDx}
&& \Big|1-F_0(x)-\sum_{i=0}^{k-1}\frac{\tcp_i}{|x|^{(i+1)\alpha}}\Big|\leq \frac{\zeta(x)}{|x|^{(1+\frac{p-\alpha}{p\alpha})\alpha}}\quad\text{for $x\rightarrow+\infty$}
\end{eqnarray}
where $\tcm_0,\tcp_0,\tcm_1,\tcp_1,\dots,\tcm_{k-1},\tcp_{k-1}$ are given in {\rm Proposition \ref{sviluppoFinf}} and \\ 
$\zeta\colon (0,+\infty)\rightarrow \R^+$ 
is a continuous, monotone decreasing function on $[B,+\infty)$ such that
\begin{equation}\label{zeta}
\int_B^{+\infty}\frac{\zeta^p(x)}{x}dx<+\infty
\end{equation}
for some $B>0$.
Then
\[
d_p(\bar{\mu}_0,\mu_\infty)<+\infty.
\]
\item [(ii)]
If $\alpha\neq1$, $\beta=-1$ {\rm[}$\beta=1$, resp.{\rm]}, suppose that \eqref{condDatoInizSx} {\rm[}$\eqref{condDatoInizDx}$, resp.{\rm]} holds true,
that $\int_{0}^{+\infty}|x|^p dF_0(x)<+\infty$ {\rm[}$\int_{-\infty}^0|x|^p dF_0(x)<+\infty$, resp.{\rm]} and $\CS(s)<0$ for some $s>\max(p,\alpha + (p-\alpha)/p)$. Then
\[
d_p(\bar{\mu}_0,\mu_\infty)<+\infty.
\]
\end{itemize}
\end{theorem}

\begin{remark}
A simple example of function $\zeta$ is $\zeta(x):=|x|^{-\veps}$ for some $\veps>0$, but
one can also take functions that decrease to infinity slower than a power, for instance
 $\zeta(x):={(\log x)^{-\frac{1+\veps}{p}}}$.

Note that  if  $p>\alpha\geq1$ then $1\leq 1+\frac{p-\alpha}{p\alpha}<2$.
Hence, in this case   $k=\lfloor 1+\frac{p-\alpha}{p\alpha}\rfloor=1$. This means that  \eqref{condDatoInizSx}-\eqref{condDatoInizDx} are
similar to the conditions that describe to so-called {\rm strong domain of attraction of an $\a$-stable law}, i.e.
\[
1-F_0(x)=\frac{c_0^+}{|x|^\alpha} + O\Big(\frac{1}{|x|^{\alpha+\delta}}\Big ),\qquad
F_0(x)=\frac{c_0^-}{|x|^\alpha} + O\Big(\frac{1}{|x|^{\alpha+\delta}}\Big),
\]
for $|x| \to \infty$ and for some $\delta>0$. See, for instance, \cite{Cramer}.
\end{remark}

\subsection{Some estimates for $\alpha =2$}\label{sec:mainresults2}
In this section we assume that $(H_0)$ holds true with $\a=2$ and we provide some estimates for the rate of convergence to equilibrium with respect to Wasserstein distances
of order $p>2$. To do so, we will employ the same inductive argument  on the order $p$  used
in the proof of Theorems \ref{alfa<1} and \ref{alfa1}. The first obstacle in this procedure is that, at the best of our knowledge, when $\a=2$, there is not a result comparable
to those of Theorems \ref{PropW-2} and \ref{PropW-3}. The only exception is for the Kac model; in this case rates of convergence both in $d_1$ and in $d_2$ are known \cite{GabettaRegazziniWM}. It would be useful to prove a result similar to Theorems \ref{PropW-2} and \ref{PropW-3} for $\a=2$ to get estimates for $d_p(\mu_t,\mu_\infty)$ $-$ with $1\leq p\leq2$ $-$ and use them as the first step of the inductive argument. The main problem is that we do not manage to give non trivial upper bounds for $d_p(\mu_t,\mu_\infty)$ with $1< p\leq2$. Indeed, the only explicit estimate that we are able to provide is given by
\begin{equation}\label{wass12}
d_p(\mu_t,\mu_\infty)\leq \Gamma_2
\end{equation}
for some positive constant $\Gamma_2$, for every $t\geq0$ and for every $1<p\leq2$. This  trivial inequality
follows since $d_p\leq d_2$ for every $1<p\leq2$ and $d_2(\mu_t,\mu_\infty)\to 0$ as $t\to+\infty$. The convergence to zero of
$d_2(\mu_t,\mu_\infty)$ is a consequence of the weak convergence of $\mu_t$ to $\mu_\infty$ supplemented by the fact that, when $\bar{\mu}_0$
satisfies the assumptions of Theorem \ref{thm2} (i.e. it has zero mean and finite variance), one has $\int_{\R}x^2\mu_t(dx)=\int_{\R}x^2\mu_\infty(dx)$ for every $t\geq0$.\newline
As for $d_1$, we obtain a non trivial bound passing through Fourier distances. Recall that for every $s>0$ the \textit{Fourier distance} $\chi_s$ (also known as weighted $\chi$-metric of order $s$) between two probability measures $\mu_1$ and $\mu_2$ on $\CB(\R)$ is defined as
\[
\chi_s(\mu_1,\mu_2):=\sup_{\xi\neq0}\frac{|\hat{\mu}_1(\xi)-\hat{\mu}_2(\xi)|}{|\xi|^s}
\]
where $\hat{\mu}_i(\xi)=\int_\R e^{i\xi x}\mu_i(dx)$ for every $\xi\in\R$ and $i=1,2$. These distances are very useful in order to easily obtain rates of convergence to equilibrium for every $\a\in(0,2]$. Indeed, one can plainly prove the following:

\begin{proposition}\label{propratechi}
Assume that $(H_0)$ holds true with $\alpha\in(0,2]$ and $p>\alpha$. If $\a\neq1$ suppose that $\bar{\mu}_0$ satisfies the hypotheses of {\rm Theorem \ref{thm2}}, while if $\a=1$ suppose that $\bar{\mu}_0$ satisfies the hypotheses of {\rm Theorem \ref{thm1}}. If $\chi_{p}(\bar{\mu}_0,\mu_\infty)<+\infty$, one has
\begin{equation*}
\chi_{p}(\mu_t,\mu_\infty)\leq \chi_{p}(\bar{\mu}_0,\mu_\infty)e^{t\CS(p)}.
\end{equation*}
\end{proposition}

In Section \ref{sec:proof2} we will prove that, for a suitable $\delta>0$, the Fourier distance of order $2+\delta$ can be used
as an upper bound for the Wasserstein distance of order $1$. Combining this fact with Proposition \ref{propratechi} with $\alpha=2$,  
we will prove the following:

\begin{theorem}\label{d1a2}
Assume that $(H_0)$ holds true with $\a=2$ and $p>2$, and that $\bar{\mu}_0$ satisfies the hypotheses of {\rm Theorem \ref{thm2}}.
Then, for every $\delta \in(0,1)$ such that $2+\delta\leq p$ and $\int_\R |x|^{2+\delta}\bar{\mu}_0(dx)<+\infty$, there exists a constant $0<C<+\infty$ such that
\begin{equation*}
d_1(\mu_t,\mu_\infty)\leq C \chi_{2+\delta}(\bar{\mu}_0,\mu_\infty)^\frac{1}{3(2+\delta)}e^{t\frac{\vphi(2+\delta)}{3}}
\end{equation*}
for every $t\geq0$ with $\chi_{2+\delta}(\bar{\mu}_0,\mu_\infty)<+\infty$.
\end{theorem}

The next theorem provides some estimates for the rate of convergence to equilibrium  with respect to Wasserstein distances of order higher than $2$.

\begin{theorem}[$\alpha=2$]\label{alfa2}
Assume that $(H_0)$ holds true with $\alpha=2$ and $p>2$, and that $\bar{\mu}_0$ satisfies the hypotheses of {\rm Theorem \ref{thm2}}. 
If $\int_\R |x|^p\bar{\mu}_0(dx)<+\infty$, 
then there exist a constants $0<C_p=C_p(\bar{\mu}_0)<+\infty$ such that for every $t\geq0$
\begin{equation*}
d_p(\mu_t,\mu_\infty)\leq \left\{
												\begin{array}{ll}
												C_{p}e^{-tR_p} & \text{if $\CS(p)\neq\frac{1}{3}\vphi(2+\veps_p)$}\\
												C_{p}te^{-tR_p} & \text{if $\CS(p)=\frac{1}{3}\vphi(2+\veps_p)$}
												\end{array}
												\right. 
\end{equation*}
with $-R_p=\max\{\vphi(p),\frac{\vphi(2+\veps_p)}{3p}\}$ and where $\veps_p\in(0,1]$ is the fractionary part of $p$.
\end{theorem}

\section{Proofs of Theorem \ref{PropW-3} and Lemma \ref{distfinita1}}\label{Sec_proof1}

We start with some useful remarks related to the probabilistic representation of the
solution. Here and in the rest of the paper $\CL(Z)$ denotes the law of a random variable $Z$.

Combining \eqref{eq.collop} and \eqref{Wild2}, it is plain to check that
\begin{equation}\label{wildconvWn}
W_{n+1}\stackrel{d}{=}LW'_{I_n}+RW''_{n+1-I_n}\qquad\text{for every $n\geq1$}
\end{equation}
where $(W'_k)_{k\geq1}$, $(W''_k)_{k\geq1}$ are independent  sequences of random variables such that
\[
W'_k\stackrel{d}{=}W''_k\stackrel{d}{=}W_k\qquad\text{for every $k\geq1$}
\]
and, in addition, $(I_n)_{n \geq 1}$ are independent random variables
uniformly distributed on $\{1,\dots,n\}$, $(W'_k)_{k\geq1}$, $(W''_k)_{k\geq1}$, $(I_n)_{n\geq1}$, $(L,R)$ are stochastically independent.

Under the assumptions of Theorem \ref{thm2} or Theorem \ref{thm1},
 let $(V_j)_{j\geq1}$ be a sequence of i.i.d. random variables with common law $\mu_\infty$ and independent of $(\beta_{j,n}:\;j=1,\dots, n)_ {n \geq 1}$.
Since $\mu_\infty$ is a stationary distribution for $Q^+$, using \eqref{Wn} with $\bar \mu_0= \mu_\infty$ and \eqref{wildconvWn}, it  immediately follows
by induction that
\begin{equation}\label{stationaryV}
\CL\Big( \sum_{j=1}^n\beta_{j,n}V_j\Big)= \mu_\infty
\end{equation}
for every $n\geq 1$.

\subsection{Proof of Lemma \ref{distfinita1}}
We begin by proving a simple lemma.

\begin{lemma}\label{lemmadp}
Consider two probability measures $\mu_1$ and $\mu_2$ on $\CB(\R)$ such that $d_p(\mu_1,\mu_2)<+\infty$ for some
$p\geq1$. Let $\tilde{\mu}_1$ be a probability measure on $\CB(\R^2)$ such that $\CL(U\cdot V)=\mu_1$ when
$(U,V)$ is distributed according to $\tilde{\mu}_1$.  Then, there exists a random vector $(X_{11},X_{12},X_2)$ such that
the law of $(X_{11},X_{12})$ is $\tilde{\mu}_1$, the law of $X_2$ is $\mu_2$ and
\[
d^p_p(\mu_1,\mu_2)=\E\Big|X_{11} X_{12}-X_2\Big|^p.
\]
\end{lemma}

\begin{proof} Let $(\overline X_1,\overline X_2)$ be an optimal coupling for $(\mu_1,\mu_2)$. If $\mu_{2|1}$ denotes the conditional law of $\overline{X}_2$
given $\overline{X}_1$, then the Disintegration Theorem leads to
\[
d_p^p(\mu_1,\mu_2)=\E\Big|\overline X_1-\overline X_2\Big|^p=\int_\R\int_\R |x_1-x_2|^p\mu_{2|1}(dx_2|x_1)\mu_1(dx_1)
\]
and, since $\int_\R|x_1-x_2|^p\mu_{2|1}(dx_2|x_1)$ is finite $\mu_1$ a.s., we can write
\[
d_p^p(\mu_1,\mu_2)=\int_{\R^2}\int_\R |x_{11} x_{12}-x_2|^p\mu_{2|1}(dx_2|x_{11} x_{12})\tilde{\mu}_1(dx_{11}, dx_{12})=\E\Big|X_{11} X_{12}-X_2\Big|^p
\]
where $(X_{11},X_{12},X_2)$ is a random vector whose probability distribution is
\[
\mu(dx_{11}, dx_{12}, dx_2):=\mu_{2|1}(dx_2|x_{11} x_{12})\tilde{\mu}_1(dx_{11}, dx_{12}).
\]
\end{proof}

Thanks to the previous lemma, we can prove Lemma \ref{distfinita1}.

\begin{proof}[Proof of Lemma \ref{distfinita1}]
From the definition of $\gamma_0$, it is clear that 
\[\lim_{R\to+\infty} \int_{(-R,R)}x \bar \mu^*_0(dx)=0\]
 and \eqref{characteristicmu*} follows from \eqref{characteristic2}.
It remains to prove the equivalence between the finiteness of $d_p(\bar{\mu}_0, \mu_\infty)$ and the one of $d_p(\bar{\mu}^*_0, \mu^*_\infty)$.
Firstly, suppose that $d_p(\bar{\mu}_0, \mu_\infty)<+\infty$.
Note that $\mu_\infty=\CL( M_\infty^{(1)} C_{\lambda,\gamma_0})$ where $ C_{\lambda,\gamma_0}$ is a
Cauchy distribution of scale parameter $\lambda=c_0 \pi$ and position $\gamma_0$, $ M_\infty^{(1)}$ has law $\nu_1$
and, finally,  $C_{\lambda,\gamma_0}$ and  $M_\infty^{(1)}$ are stochastically independent.
Hence, by Lemma \ref{lemmadp} applied with $\mu_1=\mu_\infty$, $\mu_2=\bar{\mu}_0$ and $\tilde \mu_1=\CL((C_{\lambda,\gamma_0},M_\infty^{(1)}))$,
we get the existence of a random vector $(\tilde{C}_{\lambda,\gamma_0},\tilde{M}^{(1)}_\infty,\tilde{X}_0)$
with $\CL(\tilde{X}_0)=\bar{\mu}_0$, $\CL(\tilde{C}_{\lambda,\gamma_0} \tilde{M}^{(1)}_\infty)=\mu_\infty$ and
\[
d_p(\bar{\mu}_0,\mu_\infty)=\left(\E\Big| \tilde{C}_{\lambda,\gamma_0} \tilde{M}^{(1)}_\infty- \tilde{X}_0\Big|^p\right)^\frac{1}{p}.
\]
Put $X^*_0=\tilde{X}_0-\gamma_0$,  $V^*_\infty=\Big(\tilde{C}_{\lambda,\gamma_0}-\gamma_0\Big) \tilde{M}^{(1)}_\infty$. Then, $\CL(X^*_0)=\bar{\mu}^*_0$, $\CL(V^*_\infty)=\mu^*_\infty$ and hence
\begin{eqnarray*}
d_p(\bar{\mu}^*_0, \mu^*_\infty)&\leq& \left(\E\Big|X^*_0-V^*_\infty\Big|^p\right)^\frac{1}{p}= \left(\E\Big| \tilde{X}_0-\gamma_0-\tilde{C}_{\lambda,\gamma_0}\tilde{M}^{(1)}_\infty +\gamma_0 \tilde{M}^{(1)}_\infty \Big|^p\right)^\frac{1}{p}\\
&\leq& d_p(\bar{\mu}_0,\mu_\infty)+|\gamma_0| \left(\E\Big|1-\tilde{M}^{(1)}_\infty \Big|^p\right)^\frac{1}{p}
\end{eqnarray*}
and the last term is finite since $d_p(\bar{\mu}_0,\mu_\infty)<+\infty$ and $\CS(p)<0$, which entails that $\E\Big|\tilde{M}^{(1)}_\infty\Big|^p$ is finite by Proposition \ref{Prop0-1}.\newline
Conversely, suppose that $d_p(\bar{\mu}^*_0, \mu^*_\infty)<+\infty$. Note that $\mu^*_\infty=\CL(M^{(1)}_\infty C_{\lm,0})$ and hence let $(S^*_1,M^{(1)*}_\infty,X^*_0)$ be the random vector given by Lemma \ref{lemmadp} applied with $\mu_1=\mu^*_\infty$, $\mu_2=\bar{\mu}^*_0$ and $\tilde{\mu}_1=\CL(C_{\lm,0},M^{(1)}_\infty)$. Thus, $\CL(X^*_0)=\bar{\mu}^*_0$, $\CL(S^*_1 M^{(1)*}_\infty)=\mu^*_\infty$ and
\begin{equation}\label{identitalp}
 d_p(\bar{\mu}^*_0,\mu^*_\infty)=\left(\E\Big|S^*_1 M^{(1)*}_\infty- X^*_0\Big|^p\right)^\frac{1}{p}.
\end{equation}
Put $\overline{X}_0=X^*_0+\gamma_0$, $\overline{V}_\infty=(S^*_1+\gamma_0) M^{(1)*}_\infty$. Then, $\CL(\overline{X}_0)=\bar{\mu}_0$, $\CL(\overline{V}_\infty)=\mu_\infty$ and hence
\begin{equation}\label{dpbarrata}
\begin{split}
d_p(\bar{\mu}_0, \mu_\infty)&\leq \left(\E\Big|\overline{X}_0-\overline{V}_\infty\Big|^p\right)^\frac{1}{p} =\left(\E\Big| X^*_0+\gamma_0-M^{(1)*}_\infty S^*_1-\gamma_0 M^{(1)*}_\infty \Big|^p\right)^\frac{1}{p}\\
&\leq d_p(\bar{\mu}^*_0,\mu^*_\infty)+\gamma_0 \left(\E\Big|1-M^{(1)*}_\infty \Big|^p\right)^\frac{1}{p}
\end{split}
\end{equation}
and the last term is finite since $d_p(\bar{\mu}^*_0,\mu^*_\infty)<+\infty$ and $\CS(p)<0$. This concludes the proof.
\end{proof}

\subsection{Proof of Theorem \ref{PropW-3}}
As already anticipated in the introduction of Section 3, the von Bahr-Esseen inequality has played
an important role in proving rates of convergence to equilibrium with respect to Wasserstein metrics of order $p\leq2$ in the cases in which $\a\neq1$ (i.e. Theorem \ref{PropW-2}).
For the reader's convenience we recall the statement of the von Bahr-Esseen inequality \cite{BahrEsseen}:
{\it let $Z_1,\dots,Z_n$ be independent {\rm(}real valued{\rm)} random variables such that $\E[Z_i]=0$ and
  $\E[|Z_i|^p]<+\infty$ for some $1 \leq p \leq 2$, then}
  \begin{equation}\label{vonBahrEsseen}
    \E\Big[\Big|\sum_{i=1}^n Z_i\Big|^p\Big] \leq 2\sum_{i=1}^n \E[|Z_i|^p].
  \end{equation}
In this section we establish the upper bound \eqref{main_bound} employing once again the von Bahr-Esseen inequality. To do this
we will need to prove the existence of a random vector $(\bX,\bV)$ with marginal laws, respectively, $\bar{\mu}_0$ and $\mu_\infty$ and
such that $\bX-\bV$ has finite $p$-th absolute momentum and zero mean.
These properties will be proved in Lemma \ref{lemma3} which constitutes the main tool for the proof of Theorem \ref{PropW-3}.\newline

\begin{lemma}\label{lemma3}
Assume that  $\bar{\mu}_0$ satisfies \eqref{NDA-1sym} and that \eqref{gamma} holds.
If $d_p(\bar \mu_0,\mu_\infty)<+\infty$ for some $p>1$ such that $\CS(p)<0$, then
there exists a random vector $(\bX,\bV)$ such that
\begin{itemize}
\item[(i)] $\CL(\bX)=\bar{\mu}_0$, $\CL(\bV)=\mu_\infty$;
\item[(ii)] $\E\Big|\bX-\bV\Big|^p<+\infty$;
\item[(iii)] $\E(\bX-\bV)=0$.
\end{itemize}
\end{lemma}

\begin{proof}
By Lemma \ref{distfinita1}, since  $d_p(\bar \mu_0,\mu_\infty)<+\infty$, then $d_p(\bar \mu_0^*,\mu_\infty^*)<+\infty$ and \eqref{identitalp} holds 
with  $S^*_1$ and $M^{(1)*}_\infty$ stochastically independent. Now define
\[
\bX:=X^*_0+\gamma_0,\qquad \bV:=M^{(1)*}_\infty\Big(S^*_1+\gamma_0\Big).
\]
Then (i) is trivially satisfied. As for (ii), it follows by \eqref{dpbarrata}.
It remains to prove (iii). If $c_0=0$ then, by Theorem \ref{thm1}, $\bV\stackrel{d}{=}\gamma_0 M^{(1)}_\infty$ and hence by Proposition \ref{Prop0-1} it has finite p-th moment. Thus, hypothesis $d_p(\bar \mu_0,\mu_\infty)<+\infty$ entails that $\int_\R |x|^p\bar\mu_0(dx)<+\infty$ and $\int_\R |x|\bar\mu_0(dx)<+\infty$. Combining this fact with \eqref{gamma} one has $\E(\bX)=\gamma_0$ and (iii) follows since one also has $\E(\bV)=\gamma_0$. Now, let us consider the case $c_0>0$. Thanks to (ii), $\bX-\bV$ has finite absolute momentum. Recalling that
$\E\Big(M^{(1)*}_\infty\Big)=1$, from the definition of $(\bX,\bV)$ one immediately gets
\[
\E\Big(\bX-\bV\Big)=\E\Big(X^*_0-M^{(1)*}_\infty S^*_1\Big).
\]
Denote by $F^*_0$ and $F^*_\infty$ the probability distribution functions of $\bar{\mu}^*_0$ and $\mu^*_\infty$, respectively. Let $(F^*_0)^{-1}$ and $(F^*_\infty)^{-1}$ be the corresponding quantile functions. Since $(X^*_0,M^{(1)*}_\infty S^*_1)$ is an optimal coupling for
$(\bar{\mu}_0^*,\mu_\infty^*)$, it follows that $(X^*_0,M^{(1)*}_\infty S^*_1)$ has the same law of $((F^*_0)^{-1}(U),(F^*_\infty)^{-1}(U))$ where
$U$ is a random variable with uniform distribution on $(0,1)$. Combining all these facts it easily follows that
\[
\E\Big(\bX-\bV\Big)=\lim_{n\to+\infty}\int_{\veps_n}^{1-\veps_n}\Big[(F^*_0)^{-1}(u)-(F^*_\infty)^{-1}(u)\Big] du
\]
for any sequence $(\veps_n)_{n\geq1}$ such that $\veps_n \downarrow0$ as $n\to+\infty$.
Recalling that $F^*_\infty$ is a symmetric distribution function, one gets
\[
\int_{\veps_n}^{1-\veps_n}(F^*_\infty)^{-1}(u) du=0
\]
for every $n\geq1$, which means that
\[
\E\Big(\bX-\bV\Big)=\lim_{n\to+\infty}\int_{\veps_n}^{1-\veps_n}(F^*_0)^{-1}(u) du.
\]
For the sake of notational simplicity, from now on, write $F_*$ in place of $F_0^*$.
In order to choose an appropriate sequence $(\veps_n)_{n\geq1}$, consider a real sequence $(a_n)_{n\geq1}$ such that for every $n\geq1$ $F_*^{-1}\Big(F_*(a_n)\Big)=a_n$ and $\lim_{n\to+\infty}a_n=-\infty$. Defining $\veps_n:= F_*(a_n)$ for every $n\geq1$, it is easy to prove that
\begin{equation*}
\begin{split}
\int_{\veps_n}^{1-\veps_n}& F_*^{-1}(u)du=\int_{(F^{-1}_*(\veps_n),F_*^{-1}(1-\veps_n))}x dF_*(dx)\\
&+F_*^{-1}(1-\veps_n)\Big\{1-\veps_n-F_*\Big[(F_*^{-1}(1-\veps_n)^-\Big)\Big]\Big\}:=A(n)+B(n)
\end{split}
\end{equation*}
where $F_*(x^-):=\lim_{y\to x^-}F_*(y)$. To show (iii) we have to prove that $\big(A(n)\big)_{n\geq1}$ and $\big(B(n)\big)_{n\geq1}$ are infinitesimal as $n\to+\infty$. For this purpose, we have to study the asymptotic behaviors of $F_*(x)$ as $|x|\to+\infty$ and of $F^{-1}_*(u)$ as $u\to 0^+$ or $u\to 1^-$. From \eqref{NDA-1sym}, we deduce that for every fixed $\delta\in (0,c_0)$ there exists $\bar{x}=\bar{x}(\delta)$ such that
\[
-\frac{c_0-\delta}{x}\leq F_*(x)\leq -\frac{c_0+\delta}{x} \qquad \text{for every $x\leq-\bar{x}$}
\]
\[
1-\frac{c_0+\delta}{x}\leq F_*(x)\leq 1-\frac{c_0-\delta}{x}\qquad \text{for every $x\geq\bar{x}$}.
\]
Put $A_1=A_1(\delta):=c_0+\delta$ and $A_2=A_2(\delta):=c_0-\delta$ and define two functions $G_1$, $G_2$ by
\[
G_i(x):=\left\{ \begin{array}{ll}
                -\frac{A_i}{x} & \text{if $x\leq-\bar{x}$}\\
                \frac{A_i}{\bar{x}} & \text{if $-\bar{x}< x <\bar{x}$}\\
                1-\frac{A_i}{x} & \text{if $x\geq\bar{x}$}
                \end{array}
                \right.
\]
for $i=1,2$. Then
\[
G_2(x)\leq F_*(x)\leq G_1(x)\qquad\text{for every $x\leq-\bar{x}$},
\]
\[
G_1(x)\leq F_*(x)\leq G_2(x)\qquad\text{for every $x\geq\bar{x}$}.
\]
Hence, for every $u\in(0,G_2(-\bar{x}))$,
\[
-\frac{A_1}{u}=G^{-1}_1(u)\leq F_*^{-1}(u)\leq G^{-1}_2(u)=-\frac{A_2}{u}
\]
and, for every $u\in(G_2(\bar{x}),1)$,
\[
\frac{A_2}{1-u}=G^{-1}_2(u)\leq F_*^{-1}(u)\leq G^{-1}_1(u)=\frac{A_1}{1-u}.
\]
Finally, observe that given $\delta$ and $\bar{x}$, there exists $\bar{n}=\bar{n}(\delta,\bar{x})$ such that
for every $n\geq\bar{n}$
\[
1-\veps_n\in(G_2(\bar{x}),1),\quad \veps_n\in(0,G_2(-\bar{x})),\quad \frac{A_i}{\veps_n}\geq\bar{x}.
\]
Thus, for $n\geq\bar{n}$ one has
\[
F_*^{-1}(1-\veps_n)\geq\bar{x} \qquad \text{and} \qquad F_*^{-1}(\veps_n)\leq -\bar{x}.
\]
With this information on $F_*$ and $F_*^{-1}$ we are ready to prove that $\lim_{n\to+\infty} A(n)=0$ and $\lim_{n\to+\infty} B(n)=0$. Firstly, consider $B(n)$: for every $n\geq\bar{n}$ we know that $\frac{A_2}{\veps_n}\leq F_*^{-1}(1-\veps_n)\leq \frac{A_1}{\veps_n}$ and hence, by monotonicity of $F_*$,
\[
1-\veps_n-F_*\Big(F_*^{-1}(1-\veps_n)^-\Big)\leq 1-\veps_n-F_*\Big(\Big(\frac{A_2}{\veps_n}\Big)^-\Big)\leq 1-\veps_n-G_1\Big(\frac{A_2}{\veps_n}\Big)=\veps_n\Big(\frac{A_1}{A_2}-1\Big).
\]
On the other hand,
\[
1-\veps_n-F_*\Big(F_*^{-1}(1-\veps_n)^-\Big)\geq 1-\veps_n-F_*\Big(\frac{A_1}{\veps_n}\Big)\geq 1-\veps_n-G_2\Big(\frac{A_2}{\veps_n}\Big)=\veps_n\Big(\frac{A_2}{A_1}-1\Big).
\]
Since $F_*^{-1}(1-\veps_n)\geq\bar{x}>0$, we get
\[
\veps_n\Big(\frac{A_2}{A_1}-1\Big)F_*^{-1}(1-\veps_n)\leq B(n) \leq \veps_n\Big(\frac{A_1}{A_2}-1\Big)F_*^{-1}(1-\veps_n).
\]
Note that $A_1\geq A_2$, which entails that $\frac{A_1}{A_2}\geq1$ and hence $B(n)\leq \veps_n\Big(\frac{A_1}{A_2}-1\Big)G^{-1}_1(1-\veps_n)$ and $B(n)\geq \veps_n\Big(\frac{A_2}{A_1}-1\Big)G^{-1}_1(1-\veps_n)$. This implies that
\[
A_1\Big(\frac{A_2}{A_1}-1\Big)\leq B(n)\leq A_1\Big(\frac{A_1}{A_2}-1\Big)
\]
for every $n\geq\bar{n}$, that is, by definition of $A_i$,
\[
-2\delta\leq B(n)\leq \frac{c_0+\delta}{c_0-\delta}2\delta
\]
for every $\delta>0$ and for every $n\geq\bar{n}(\delta,\bar{x})$. Hence, $\lim_{n\to+\infty}B(n)=0$.\newline
Finally, consider $A(n)$. Recall that, from Lemma \ref{distfinita1},
\begin{equation}\label{gamma0=0}
\lim_{R\to+\infty}\int_{(-R,R)}x dF_*(x)=0.
\end{equation}
We will take advantage of this property by splitting the integral $A(n)$ into two integrals, one of them over a symmetric interval about the origin. Fix $n\geq\bar{n}$. If $-F_*^{-1}(\veps_n)\leq F_*^{-1}(1-\veps_n)$, then
\begin{equation}\label{split1}
\begin{split}
\int_{(F_*^{-1}(\veps_n),F_*^{-1}(1-\veps_n))}x dF_*(dx)&=\int_{(F_*^{-1}(\veps_n),-F_*^{-1}(\veps_n))}x dF_*(x)\\
&+\int_{[-F_*^{-1}(\veps_n),F_*^{-1}(1-\veps_n))}x dF_*(dx).
\end{split}
\end{equation}
On the other hand, if $-F_*^{-1}(\veps_n)\geq F_*^{-1}(1-\veps_n)$, then
\begin{equation}\label{split2}
\begin{split}
\int_{(F_*^{-1}(\veps_n),F_*^{-1}(1-\veps_n))}x dF_*(dx)&=\int_{(F_*^{-1}(\veps_n),-F_*^{-1}(\veps_n))}x dF_*(x)\\
&-\int_{[F_*^{-1}(1-\veps_n),-F_*^{-1}(\veps_n))}x dF_*(dx).
\end{split}
\end{equation}
Thanks to \eqref{gamma0=0}, the first integrals on the right hand side of both \eqref{split1} and \eqref{split2} converge to zero when $n\to+\infty$. As concerns the second integrals, recall that for every $n\geq\bar{n}$ one has
\[
\frac{A_2}{\veps_n}\leq -F_*^{-1}(\veps_n)\leq \frac{A_1}{\veps_n} \quad\text{and}\quad \frac{A_2}{\veps_n}\leq F_*^{-1}(1-\veps_n)\leq \frac{A_1}{\veps_n}
\]
and hence
\[
0\leq \int_{[-F_*^{-1}(\veps_n),F_*^{-1}(1-\veps_n))}x dF_*(x)\leq \int_{[\frac{A_2}{\veps_n},\frac{A_1}{\veps_n})}x dF_*(x),
\]
\[
0\leq \int_{[F_*^{-1}(1-\veps_n),-F_*^{-1}(\veps_n))}x dF_*(x)\leq \int_{[\frac{A_2}{\veps_n},\frac{A_1}{\veps_n})}x dF_*(x).
\]
The positiveness can be obtained by further increasing $\bar{n}$, if needed. Thus, in order to prove that the second integrals in \eqref{split1} and \eqref{split2} converge to zero as $n\to+\infty$
we have to show that $\int_{[\frac{A_2}{\veps_n},\frac{A_1}{\veps_n})}x dF_*(x)$ converge to zero as $n\to+\infty$. By partial integration and using the estimates of $F_*$ with $G_1$ and $G_2$ we get
\begin{eqnarray*}
\int_{[\frac{A_2}{\veps_n},\frac{A_1}{\veps_n})}x dF_*(x)&\leq& F_*\Big[\Big(\frac{A_1}{\veps_n}\Big)^-\Big]\frac{A_1}{\veps_n}-F_*\Big[\Big(\frac{A_2}{\veps_n}\Big)^-\Big]\frac{A_2}{\veps_n}
- \int_{(\frac{A_2}{\veps_n},\frac{A_1}{\veps_n}]} F_*(x) dx\\
&\leq& G_2\Big(\frac{A_1}{\veps_n}\Big)\frac{A_1}{\veps_n}-G_1\Big(\frac{A_2}{\veps_n}\Big)\frac{A_2}{\veps_n}
\int_{(\frac{A_2}{\veps_n},\frac{A_1}{\veps_n}]} \Big[1-\frac{A_1}{x}\Big] dx\\
&=& \Big(1-\frac{A_2}{A_1}\veps_n\Big)\frac{A_1}{\veps_n}-\Big(1-\frac{A_1}{A_2}\veps_n\Big)\frac{A_2}{\veps_n}-\Big(\frac{A_1}{\veps_n}-\frac{A_2}{\veps_n}\Big)\\
&+&A_1\log\frac{A_1}{A_2}\\
&=& 2\delta+(c_0+\delta)\log\Big(1+\frac{2\delta}{c_0-\delta}\Big).
\end{eqnarray*}
Thanks to the arbitrariness of $\delta>0$, this entails that the second integrals in \eqref{split1} and \eqref{split2} converge to zero as $n\to+\infty$ and hence $\lim_{n\to+\infty}A(n)=0$. This implies (iii) and concludes the proof.
\end{proof}

\begin{proof}[Proof of Theorem \ref{PropW-3}.]
Let $(\bX,\bV)$ be the random vector given by Lemma \ref{lemma3}. Consider a sequence $(X_j,V_j)_{j\geq1}$ of i.i.d. random vectors with the same distribution of $(\bX,\bV)$ and
such that $(X_j,V_j)_{j\geq1}$ is stochastically independent of $\CB=\sigma\{(\beta_{j,n}: j=1,\dots,n)_{n\geq1}\}$.
By \eqref{stationaryV}  we already know that
$\sum_{j=1}^{N_t}\beta_{j,N_t}V_j$ has probability distribution $\mu_\infty$.
Now, for every $n\geq 0$, denote by $\mu_n$ the law of the random variable
$W_{n+1}$, defined in \eqref{Wn}.
Hence, by convexity
\[
\begin{split}
d^p_p(\mu_t,\mu_\infty) & \leq\sum_{n\geq1}e^{-t}(1-e^{-t})^{n-1}d_p^p(\mu_{n-1},\mu_\infty)\\
&\leq \sum_{n\geq1} e^{-t}(1-e^{-t})^{n-1} \E\Big|\sum_{j=1}^n\beta_{j,n}(X_j-V_j)\Big|^p.\\
\end{split}
\]
Since $\E|X_j-V_j|^p<+\infty$ and $\E(X_j-V_j)=0$, we can make use of the von Bahr-Esseen inequality \eqref{vonBahrEsseen} $-$ conditionally to $\CB$ $-$ and get
\[
\begin{split}
\E\Big|\sum_{j=1}^n\beta_{j,n}(X_j-V_j)\Big|^p
& =\E\Big[ \E \Big|\sum_{j=1}^n\beta_{j,n}(X_j-V_j)\Big|^p\Big|\CB\Big ] \\
& \leq 2\E\Big(\sum_{j=1}^n \beta^p_{j,n} \E\Big[ |X_j-V_j|^p\Big|\CB \Big] \Big)=2\E|\bX-\bV|^p\E\Big(\sum_{j=1}^n\beta^p_{j,n}\Big).\\
\end{split}
\]
From Lemma 2 in \cite{BassLadMatth10}, one has
\begin{equation}\label{momentiM}
\E\Big(\sum_{j=1}^n\beta^p_{j,n}\Big)= \frac{\Gamma(n+\CS(p))}{\Gamma(n)\Gamma(\CS(p)+1)}
\end{equation}
and hence, 
 recalling that
for every $\gamma>-1$ and $0<u<1$
\begin{equation}\label{serie-gamma}
\sum_{n=1}^{+\infty}\frac{\Gamma (\gamma+n)}{\Gamma (n)\Gamma (\gamma+1)}(1-u)^{n-1}=u^{-(\gamma+1)},
\end{equation}
one gets
\[
d_p(\mu_t,\mu_\infty)\leq C_p e^{t\frac{\CS(p)}{p}}
\]
with $C_p:=\left(2\E|\bX-\bV|^p\right)^\frac{1}{p}$.
\end{proof}

\section{Proof of Theorems \ref{alfa<1} and \ref{alfa1}}\label{section5}
In this section we will prove the exponential rates of convergence to equilibrium which have been presented in Section \ref{sec:mainresults}.
We will develop in details only the proof of Theorem \ref{alfa<1} since Theorem \ref{alfa1} can be proved in a very similar way with slight adaptations.
As already anticipated, both Theorems \ref{alfa<1} and \ref{alfa1} descend from an inductive argument $-$ applied to the order of the Wasserstein distance $-$ supplemented by
the probabilistic representation of the solution of \eqref{eq.1} briefly recalled in Section \ref{probrep}.
Recall that for every $n\geq0$, $\mu_n$ is the law of the random variable $W_{n+1}$ introduced in \eqref{Wn}.

We start by proving two simple lemmata:

\begin{lemma}\label{rho-contlim}
Assume that $(H_0)$ holds true for some $p>\a$ such that $d_p(\bar{\mu}_0,\mu_\infty)<+\infty$. Then the function
\begin{equation}\label{serierho}
t\mapsto\sigma_t(p):=\sum_{n\geq1}(1-e^{-t})^{n-1}d^p_p(\mu_{n-1},\mu_\infty)
\end{equation}
is continuous and bounded on every interval $[0,T]$.
\end{lemma}

\begin{proof}
For every fixed $t\in[0,T]$, we have to show that the series in \eqref{serierho} converges. In view of the hypothesis $d_p(\bar{\mu}_0,\mu_\infty)<+\infty$, there exists a random vector $(X_0,V_\infty)$ such that $\CL(X_0)=\bar{\mu}_0$, $\CL(V_\infty)=\mu_\infty$ and $d_p(\bar{\mu}_0,\mu_\infty)=\left(\E|X_0-V_\infty|^p\right)^\frac{1}{p}$. Consider a sequence $(X_j,V_j)_{j\geq1}$ of i.i.d. random vectors distributed as $(X_0,V_\infty)$ and independent of $(\beta_{j,n}:j=1,\dots,n)_{n\geq1}$.
By \eqref{stationaryV}, we have that $\sum_{j=1}^n\beta_{j,n} V_j  \stackrel{d}{=}V_\infty $; hence
\[
\begin{split}
\sum_{n\geq1}(1-e^{-t})^{n-1}d^p_p(\mu_{n-1},\mu_\infty)& \leq  \sum_{n\geq1}(1-e^{-t})^{n-1}\E\Big|\sum_{j=1}^n\beta_{j,n}(X_j-V_j)\Big|^p\\
& \leq \sum_{n\geq1}(1-e^{-t})^{n-1}n^{\max(p,1)-1}\E\Big(\sum_{j=1}^n\beta^p_{j,n}\Big)d^p_p(\bar{\mu}_0,\mu_\infty). \\
\end{split}
\]
By \eqref{momentiM}, we conclude that the series in \eqref{serierho} converges. Thus, the function defined in \eqref{serierho} is bounded and continuous at every $t\in[0,T]$.
\end{proof}

\begin{lemma}\label{unicominimo}
Let $\vphi$ be the function defined in \eqref{defPhi}. Assume that $(H_0)$ holds true for some $p>\alpha$ and define 
$\bar{p}:=\sup\{q>\alpha:\vphi(q)<0\}$. Then the function $\varphi$ is continuous on $[\a,\bar p]$ if $\bar p<+\infty$ and
$\varphi(\bar p)\leq 0$ and  on $[\a,\bar p)$ in all the other cases.  Moreover one of the following is true:
\begin{itemize}
 \item[(i)] the function $\vphi$ is strictly decreasing on $(\a,\bar p)$;
\item[(ii)] there exists a point $p_0<\bar p$  such that the function $\vphi$ is strictly decreasing on $(\a,p_0)$ and
strictly increasing on $(p_0,\bar p)$; 
\end{itemize}
\end{lemma}

\begin{proof}
First of all, by the dominated convergence theorem, one proves that $q\mapsto\CS(q)$ is continuous on its domain. Moreover, one can easily show that for every $q$ belonging to the interior of the domain of $\CS$
\begin{equation}\label{derivataS}
\frac{d}{dq}\CS(q)=\E\Big[\frac{d}{dq}\Big(L^q+R^q\Big)\Big]=\E\Big(L^q \log L+R^q \log R\Big),
\end{equation}
and
\begin{equation}\label{derivataSseconda}
\frac{d^2}{dq^2}\CS(q)=\E\Big[L^q(\log L)^2+R^q(\log R)^2\Big].
\end{equation}
Now consider $\vphi$ on the interval $(0,\bar{p})$; this interval is obviously included in the interior of the domain of $\CS$ and, therefore, $\vphi$ is differentiable on $(0,\bar{p})$ and
\[
\vphi'(q)=\frac{\CS'(q)q-\CS(q)}{q^2}.
\]
Now we claim that there is at most one point $p_0\in(0,\bar{p})$ such that $\vphi'(p_0)=0$, i.e. $\CS'(p_0)p_0-\CS(p_0)=0$. Computing the derivative one gets
\[
\frac{d}{dq}\Big(\CS'(q)q-\CS(q)\Big)=q\CS''(q)
\]
which, from \eqref{derivataSseconda}, is strictly positive since $\P\{(L,R)\in\{0,1\}^2\}<1$ (see \eqref{H1}). Thus, 
$q\mapsto \CS'(q)q-\CS(q)$ is a strictly increasing function and the claim follows if we show that
\begin{equation}\label{limsup}
\limsup_{q\to 0^+}\Big(\CS'(q)q-\CS(q)\Big)<0.
\end{equation}
To this end, fix $q\in(0,\a]$ and note that
\[
L^q\I_{\{L>0\}}\leq 1 \I_{\{0<L\leq1\}}+L^{\a} \I_{\{L>1\}}
\]
and that the right hand side is integrable; an analogous fact obviously holds for $R$. Then, by dominated convergence theorem,
\begin{eqnarray*}
\lim_{q\to0^+}\CS(q)&=&\lim_{q\to0^+}\E\Big[L^q\I_{\{L>0\}}+R^q\I_{\{R>0\}}\Big]-1\\
&=& \P\{L>0\}+\P\{R>0\}-1
\end{eqnarray*}
which, by hypothesis \eqref{S(0)}, is strictly positive and hence $-\lim_{q\to0^+}\CS(q)<0$.\newline
On the other hand, since $\CS$ is convex and $(H_0)$ holds, then $q\CS'(q)<0$ for every $q\in(0,\alpha]$; therefore
\[
\limsup_{q\to0^+}q\CS'(q)\leq0
\]
and hence \eqref{limsup} holds. Thus, we obtain that there is at most one point $p_0$ such that $\vphi'(p_0)=0$. 
The thesis follows since $0=\CS(\a)>\CS(p)$ and hence $0=\varphi(\a)>\vphi(p)$.
\end{proof}

Here we prove a proposition that will give the fundamental tools for the inductive argument that we will use in the proofs of Theorems \ref{alfa<1}-\ref{alfa1}-\ref{alfa2}.

\begin{proposition}
Assume that the assumptions of {\rm Lemma \ref{rho-contlim}} are in force. Consider the function $\sigma_t$ defined in \eqref{serierho} and a real number $q$ such that $1<q\leq p$. Then
\begin{itemize}
\item if $1<q<2$ then 
\begin{equation}\label{sigma<2}
\sigma_t(q)\leq d^q_q(\bar{\mu}_0,\mu_\infty)e^{t\CS(q)}+e^{t\CS(q)}B_{q}\int_0^t e^{-\tau\CS(q)}\sigma^{q}_\tau(1)d\tau;
\end{equation}\\
\item if $p\geq2$ and $q\geq2$ then
\begin{equation}\label{sigma>2}
\sigma_t(q)\leq d^q_q(\bar{\mu}_0,\mu_\infty)e^{t\CS(q)}+e^{t\CS(q)}B_q\int_0^t e^{-\tau\CS(q)}\sigma_\tau(1)\sigma_\tau(q-1)d\tau.
\end{equation}
\end{itemize}
for a suitable constant $B_q$. Moreover, for every $s\geq1$ one has
\begin{equation}\label{dp<sigma}
d^s_s(\mu_t,\mu_\infty)\leq \sigma_t(s).
\end{equation}
\end{proposition}

\begin{proof}
Statement \eqref{dp<sigma} is trivial since, by Jensen's inequality, one has
\[
d^s_s(\mu_t,\mu_\infty)\leq\sum_{n\geq0}e^{-t}(1-e^{-t})^n d^s_s(\mu_n,\mu_\infty)=\sigma_t(s).
\]
Now we prove \eqref{sigma<2} and \eqref{sigma>2}. Consider two stochastically independent sequences $(W'_k,V'_k)_{k\geq1}$, $(W''_k,V_k'')_{k\geq1}$,
such that $V'_k\stackrel{d}{=}V''_k\stackrel{d}{=}V_\infty$; additionally, suppose that $(W'_k,V'_k)_{k\geq1}$, $(W''_k,V_k'')_{k\geq1}$, 
are stochastically independent of $((L_n,R_n))_{n\geq1}$, $(I_n)_{n\geq1}$, $(N_t)_{t\geq0}$ and, for every $k\geq1$, $(W'_k,V'_k)$ and $(W''_k,V''_k)$ 
are optimal couplings for $d_s(\mu_{k-1},\mu_\infty)$ for every $s\geq1$. Let us specify that we can always find such random variables since, having defined for every $x\in\R$
\[
\begin{array}{ll}
& F_k(x):=\mu_k((-\infty,x])\\
& F_\infty(x):=\mu_\infty((-\infty,x]),
\end{array}
\]
it suffices to choose $W'_k=F^{-1}_{k-1}(U'_k)$, $W''_k=F^{-1}_{k-1}(U''_k)$,$V'_k=F^{-1}_\infty(U'_k)$, $V''_k=F^{-1}_\infty(U_k'')$ with $(U'_k)_{k\geq1}$, $(U''_k)_{k\geq1}$, 
i.i.d. random variables uniformly distributed on $(0,1)$.
Recall also the following fact: if $a,b\in\R^+$ and $q>1$, then
\begin{equation}\label{sommep}
(a+b)^q\leq a^q+b^q+c_q(a^{q-1}b+ab^{q-1})
\end{equation}
with $c_q:=q$ if $q \in[2,3]$ and $c_q:=q2^{q-3}$ otherwise; see, e.g., Lemma 3.1 in \cite{MatthesToscani}. Now put $\Delta_{n+1}(s):=d^s_s(\mu_n,\mu_\infty)$ for every $n\geq0$ and $s\geq1$. 
Thanks to the independence of $(W'_k,V'_k)$ and $(W''_k,V''_k)$, ($\ref{wildconvWn}$) and ($\ref{sommep}$)
lead to
\[
\begin{split}
\rho_t(q):&=  e^t\sigma_t(q)\leq\Delta_1(q)+\sum_{n\geq1}(1-e^{-t})^n\E\Big|L(W'_{I_n}-V'_{I_n})+R(W''_{n+1-I_n}-V''_{n+1-I_n})\Big|^q\\
&\leq \Delta_1(q)+\sum_{n\geq1}(1-e^{-t})^n\Big[\E\Big(L^q|W'_{I_n}-V'_{I_n}|^q+R^q|W''_{n+1-I_n}-V''_{n+1-I_n}|^q\Big)\\
& \qquad +c_q\E\Big(L^{q-1}R|W'_{I_n}-V'_{I_n}|^{q-1}|W''_{n+1-I_n}-V''_{n+1-I_n}|\\
&\qquad  +LR^{q-1}|W'_{I_n}-V'_{I_n}||W''_{n+1-I_n}-V''_{n+1-I_n}|^{q-1}\Big)\Big]\\
&= \Delta_1(q)+\sum_{n\geq1}\frac{(1-e^{-t})^n}{n}\sum_{k=1}^n\Big[\Big(\E(L^q)\E\Big|W'_k-V'_k\Big|^q+\E(R^q)\E\Big|W''_{n+1-k}-V''_{n+1-k}\Big|^q \\
& \qquad +c_q\Big(\E(L^{q-1}R)\E\Big|W'_k-V'_k\Big|^{q-1}\E\Big|W''_{n+1-k}-V''_{n+1-k}\Big|\\
& \qquad +\E(LR^{q-1})\E\Big|W'_k-V'_k\Big|\E\Big|W''_{n+1-k}-V''_{n+1-k}\Big|^{q-1}\Big)\Big]\\
&= \Delta_1(q)+\sum_{n\geq1}\frac{(1-e^{-t})^n}{n}\sum_{k=1}^n\Big(\E(L^q+R^q)\E\Big|W'_k-V'_k\Big|^q\\
& \qquad +c_q\E(L^{q-1}R+LR^{q-1})\E\Big|W'_k-V'_k\Big|^{q-1}\E\Big|W''_{n+1-k}-V''_{n+1-k}\Big|\Big).\\
\end{split}
\]
Recalling that $(W'_k,V'_k)$ and $(W''_k,V''_k)$ have been defined as optimal couplings for $d_s(\mu_{k-1},\mu_\infty)$ for every $s\geq1$ and putting $\lambda_q:=\E(L^q+R^q)=\CS(q)+1$, $B_q:=c_q\E(L^{q-1}R+LR^{q-1})$, one has
\begin{equation}\label{ro1}
\rho_t(q)\leq\Delta_1(q)+\sum_{n\geq1}\frac{(1-e^{-t})^n}{n}\sum_{k=1}^n\Big(\lmq\Delta_k(q)+B_q\E\Big|W'_k-V'_k\Big|^{q-1}\Delta_{n+1-k}(1)\Big).
\end{equation}
At this stage, we have to distinguish two different situations, i.e. $1<q<2$ or $q\geq2$ (which is possible if $p\geq2$). The reason for this distinction 
lies in the fact that if $q\geq2$ then $q-1\geq1$ and hence (by definition of $(W'_k,V_k')_{k\geq1}$) $\E|W'_k-V'_k|^{q-1}=d^{q-1}_{q-1}(\mu_{k-1},\mu_\infty)$ while, 
if $q<2$ then $q-1<1$ and $\E|W'_k-V'_k|^{q-1}$ is not equal to $d_{q-1}(\mu_{k-1},\mu_\infty)$. 
We begin to consider the case $q\geq2$: as already noticed, $\E|W'_k-V'_k|^{q-1}=\Delta_k(q-1)$ and hence, from \eqref{ro1}, one has
\[
\begin{split}
\rho_t(q) &\leq  \Delta_1(q)+\sum_{n\geq1}\frac{(1-e^{-t})^n}{n}\sum_{k=1}^n \Big(\lmq\Delta_k(q)+B_q\Delta_k(q-1)\Delta_{n+1-k}(1)\Big)\\
&= \Delta_1(q)+\sum_{k\geq1}\sum_{j\geq0} \frac{(1-e^{-t})^{j+k}}{j+k}\Big(\lmq\Delta_k(q)+B_q\Delta_k(q-1)\Delta_{j+1}(1)\Big)\\
&= \Delta_1(q)+\sum_{k\geq1}\sum_{j\geq0} \int_{0}^t (1-e^{-\tau})^{j+k-1} 
e^{-\tau}d\tau \Big(\lmq\Delta_k(q)+B_q\Delta_k(q-1)\Delta_{j+1}(1)\Big)\\
\end{split}
\]
\[
\begin{split}
\phantom{\rho_t(q)}
&= \Delta_1(q)+\int_{0}^t\Big[\sum_{k\geq1}\Big(\sum_{j\geq0}e^{-\tau}(1-e^{-\tau})^j\Big)(1-e^{-\tau})^{k-1}\lmq\Delta_k(q)\\
& \quad +\sum_{k\geq1}\Big(\sum_{j\geq0}(1-e^{-\tau})^je^{-\tau}\Delta_{j+1}(1)\Big)B_q(1-e^{-\tau})^{k-1}\Delta_k(q-1)\Big]d\tau\\
&= \Delta_1(q)+\int_{0}^{t}\Big(\lmq\rho_\tau(q)+B_q\sigma_\tau(1)\rho_\tau(q-1)\Big)d\tau,\\
\end{split}
\]
which means that
\begin{equation*}
\rho_t(q)\leq \Delta_1(q)+\lmq\int_{0}^{t}\rho_\tau(q)d\tau+B_q\int_0^t e^\tau\sigma_\tau(1)\sigma_\tau(q-1)d\tau.
\end{equation*}
Thanks to Gronwall Lemma (whose applicability is guaranteed by Lemma $\ref{rho-contlim}$), it follows that
\begin{equation*}
\rho_t(q)\leq \Delta_1(q)e^{\lmq t}+B_q\int_0^t e^{\lmq(t-\tau)}e^\tau\sigma_\tau(1)\sigma_\tau(q-1)d\tau.
\end{equation*}
Hence, for any $q\geq2$,
\begin{equation}\label{sigmagronwall}
\sigma_t(q)\leq \Delta_1(q)e^{-t(1-\lmq)}+e^{-t(1-\lmq)}B_q\int_0^t e^{\tau(1-\lmq)}\sigma_\tau(1)\sigma_\tau(q-1)d\tau
\end{equation}
which gives \eqref{sigma>2}.
On the other hand, if $1<q<2$ then, by Jensen's inequality, $\E|W'_k-V'_k|^{q-1}\leq \Big(\E|W'_k-V'_k|\Big)^{q-1}=\Delta^{q-1}_k(1)$ and, with the same technique used to get \eqref{sigmagronwall} from \eqref{ro1}, one can easily obtain
\begin{equation*}
\sigma_t(q)\leq \Delta_1(q)e^{-t(1-\lambda_{q})}+e^{-t(1-\lambda_{q})}B_{q}\int_0^t e^{\tau(1-\lambda_{q})}\sigma^{q}_\tau(1)d\tau
\end{equation*}
which gives \eqref{sigma<2}.
\end{proof}

We are now ready to prove Theorem \ref{alfa<1} and Theorem \ref{alfa1}.

\begin{proof}[Proof of Theorem \ref{alfa<1}]
From the hypotheses one knows that $\CS(\alpha)=0$, $\CS(p)<0$ and $p>1$; hence, thanks to the convexity of $\CS$, it is clear that $\CS(1)<0$. Thus, 
from the proof of Theorem 5 in \cite{BassLadMatth10} we have that
\begin{equation}\label{sigma1}
\sigma_t(1)\leq d_1(\bar{\mu}_0,\mu_\infty) e^{t\CS(1)}.
\end{equation}
Define the integer $k_p\geq1$ and the real number $\veps_p \in(0,1]$ such that $p=k_p+\veps_p$, i.e. $k_p$ and $\veps_p$ are, respectively, the integer and the fractionary part of $p$. \newline 
\newline
\underline{Step 1.} Let us assume that
\begin{equation}\label{H}
\vphi(i+\veps_p)\neq \vphi(1)\qquad\text{for every $i=1,\dots,k_p$.}
\end{equation}
Under this assumption we show by mathematical induction that
\begin{equation}\label{induzionesigma}
\sigma_t(i+\veps_p)\leq C_{i+\veps_p}e^{-tK(i+\veps_p)}\quad\text{for every $i=1,\dots,k_p$}
\end{equation}
for suitable constants $0<C_{i+\veps_p}<+\infty$ with
\begin{equation}\label{Kp}
-K(i+\veps_p):=\max\{\CS(i+\veps_p),\CS(1)(i+\veps_p)\}.
\end{equation}
Note that \eqref{induzionesigma}-\eqref{Kp} for $i=k_p$, supplemented by \eqref{dp<sigma}, gives \eqref{ratea<1}.
In order to prove \eqref{induzionesigma}-\eqref{Kp} for $i=1$,
it suffices to combine \eqref{sigma<2} and \eqref{sigma1}  to get
\begin{equation}\label{induzione1}
\begin{split}
\sigma_t(1+\veps_p)\leq &\Delta_1(1+\veps_p)e^{t\CS(1+\veps_p)}\\
&+e^{t\CS(1+\veps_p)}B_{1+\veps_p}\int_0^t C_1 e^{\tau[-\CS(1+\veps_p)+\CS(1)(1+\veps_p)]}d\tau
\end{split}
\end{equation}
with $C_1:=d_1(\bar{\mu}_0,\mu_\infty)^{1+\veps_p}$. By hypothesis \eqref{H} it follows that $-\CS(1+\veps_p)+\CS(1)(1+\veps_p)\neq0$ and hence, solving the integral, one obtains that
\[
\sigma_t(1+\veps_p)\leq C_{1+\veps_p}e^{-tK(1+\veps_p)}
\]
for a suitable constant $0<C_{1+\veps_p}<+\infty$. 
This proves  \eqref{induzionesigma}-\eqref{Kp} for $i=1$.
If $k_p=1$, there is nothing else to be proved.
If $k_p\geq2$ we proceed by induction. Assuming that  \eqref{induzionesigma}-\eqref{Kp}  hold true for every $i=1,\dots,j-1$ ($2\leq j\leq k_p$),  we show that they hold for $i=j$.
By \eqref{sigma>2} and \eqref{sigma1} we have that
\begin{equation}\label{j+eps}
\begin{split}
\sigma_t(j+\veps_p) & \leq \Delta_1(j+\veps_p)e^{t\CS(j+\veps_p)}\\
& +B_{j+\veps_p}C_1 C_{j-1+\veps_p}e^{t\CS(j+\veps_p)}\int_0^t e^{\tau[-\CS(j+\veps_p)+\CS(1)-K(j-1+\veps_p)]}d\tau. \\
\end{split}
\end{equation}
Let us now show that the exponent in the integral above is non-zero, i.e. $-\CS(j+\veps_p)+\CS(1)-K(j-1+\veps_p)\neq0$ whatever is the value of $\vphi(j-1+\veps_p)$. If $\vphi(j-1+\veps_p)>\vphi(1)$ then, by Lemma \ref{unicominimo}, $\vphi(j+\veps_p)\geq \vphi(j-1+\veps_p)$ and hence
\[
\begin{split}
-\CS(j+\veps_p) &+\CS(1)-K(j-1+\veps_p)  < -\CS(j+\veps_p)+\vphi(j-1+\veps_p)\\
& +\vphi(j-1+\veps_p)(j-1+\veps_p)
 = -\CS(j+\veps_p)+\vphi(j-1+\veps_p)(j+\veps_p)\\
& \leq -\CS(j+\veps_p)+\vphi(j+\veps_p)(j+\veps_p)=0.
\end{split}
\]
On the other hand, if $\vphi(j-1+\veps_p)<\vphi(1)$, then
\[
\begin{split} 
-\CS(j+\veps_p)& +\CS(1)-K(j-1+\veps_p) =-\CS(j+\veps_p) +\CS(1)\\
& +(j-1+\veps_p)\CS(1)= -\CS(j+\veps_p) +\CS(1)(j+\veps_p)\neq0
\end{split}
\]
by assumption \eqref{H}.
Having proved that the exponent in the integral in \eqref{j+eps} is non-zero, an explicit integration gives \eqref{induzionesigma}-\eqref{Kp} provided that the equality
\begin{equation}\label{59bis}
\max\{\CS(j+\veps_p),\CS(1)-K(j-1+\veps_p)\}=-K(j+\veps_p)
\end{equation}
holds. Thus, let us prove this equality. If $\vphi(j+\veps_p)<\vphi(1)$ then, by Lemma \ref{unicominimo}, $\vphi(j-1+\veps_p)<\vphi(1)$ and by the inductive step $-K(j-1+\veps_p)=(j-1+\veps_p)\CS(1)$. Hence, 
\[
\max\{\CS(j+\veps_p),\CS(1)-K(j-1+\veps_p)\}=\max\{\CS(j+\veps_p),(j+\veps_p)\CS(1)\}
\]
which is \eqref{59bis}.
On the other hand, let us assume that $\vphi(j+\veps_p)>\vphi(1)$. We need to treat separately two cases. If $\vphi(j-1+\veps_p)<\vphi(1)$ then $-K(j-1+\veps_p)=(j-1+\veps_p)\CS(1)$ and hence \eqref{59bis} holds. If $\vphi(j-1+\veps_p)>\vphi(1)$ then \[
\max\{\CS(j+\veps_p),\CS(1)-K(j-1+\veps_p)\}=\max\{\CS(j+\veps_p),\CS(1)+\CS(j-1+\veps_p)\}
\] 
and, by Lemma \ref{unicominimo}, $\vphi(j+\veps_p)\geq\vphi(j-1+\veps_p)>\vphi(1)$. Hence
\[
\CS(1)+\CS(j-1+\veps_p)<\vphi(j-1+\veps_p)(j+\veps_p)\leq \CS(j+\veps_p).
\] 
This shows that 
\[
\max\{\CS(j+\veps_p),\CS(1)-K(j-1+\veps_p)\}=\CS(j+\veps_p)=\max\{\CS(j+\veps_p),\CS(1)(j+\veps_p)\}
\]
which is \eqref{59bis}.
This concludes the proof when \eqref{H} holds.\newline
\newline
\underline{Step 2.} 
Let us now assume that  $\vphi(p)=\vphi(1)$.  By Lemma \ref{unicominimo} it follows that
 $\vphi(j+\veps_p)<\vphi(1)$ for every $j=1,\dots,k_p-1$. Hence, the proof can be developed by induction as in Step 1 for $j=1,\dots,k_p-1$ and in particular $-K(k_p-1+\veps_p)=-K(p-1)=(p-1)\CS(1)$. Using this equality in \eqref{j+eps}, one gets
\[
\begin{split}
\sigma_t(p) & \leq \Delta_1(p)e^{t\CS(p)} +B_{p}C_1 C_{p-1}e^{t\CS(p)}\int_0^t e^{-p\tau[\vphi(p)-\vphi(1)]}d\tau \\
&= \Delta_1(p)e^{t\CS(p)}+B_{p}C_1 C_{p-1}te^{t\CS(p)}\leq C_p t e^{t\CS(p)}
\end{split}
\]
which gives \eqref{ratea<1} when $\vphi(p)=\vphi(1)$.\newline
\newline
\underline{Step 3.}
It remains to consider the case in which there exists $i^*\in\{1,\dots,k_p-1\}$ such that $\vphi(i^*+\veps_p)=\vphi(1)$. Arguing as in Step 1, one proves that \eqref{induzionesigma}-\eqref{Kp} hold for $i=1,\dots,i^*-1$. Moreover, arguing as in Step 2 one gets 
\[
\sigma_t(i^*+\veps_p)\leq C_{i^*+\veps_p} t e^{t\CS(i^*+\veps_p)}.
\]
Now we prove that \eqref{induzionesigma}-\eqref{Kp} hold for $i=i^*+1$. By \eqref{sigma>2} and the above inequality one gets
\begin{equation}\label{eqeta}
\begin{split}
\sigma_t(i^*+1+\veps_p) & \leq \Delta_1(i^*+1+\veps_p)e^{t\CS(i^*+1+\veps_p)} \\
& + B_{i^*+1+\veps_p}C_1 C_{i^*+\veps_p}e^{t\CS(i^*+1+\veps_p)}\int_0^t \tau e^{\tau[-\CS(i^*+1+\veps_p)+\CS(1)+\CS(i^*+\veps_p)]}d\tau \\
& \leq \Delta_1(i^*+1+\veps_p)e^{t\CS(i^*+1+\veps_p)}\\
&+ B_{i^*+1+\veps_p}C_1 C_{i^*+\veps_p}e^{t\CS(i^*+1+\veps_p)}\int_0^t  e^{\tau[-\CS(i^*+1+\veps_p)+\CS(1)+\CS(i^*+\veps_p)+\eta]}d\tau
\end{split}
\end{equation}
for every $\eta>0$. Moreover, one has
\[
\begin{split}
-\CS(i^*+1+\veps_p)&+ \CS(1)+\CS(i^*+\veps_p) = -\CS(i^*+1+\veps_p)+\vphi(i^*+\veps_p)+\CS(i^*+\veps_p)\\
& \qquad\text{[since $\vphi(i^*+\veps_p)=\vphi(1)=\CS(1)$]}\\
& = -\CS(i^*+1+\veps_p)+\vphi(i^*+\veps_p)(i^*+1+\veps_p) \\
& < -\CS(i^*+1+\veps_p)+\vphi(i^*+1+\veps_p)(i^*+1+\veps_p)=0\\
& \qquad \text{[since, by Lemma \ref{unicominimo}, $\vphi(i^*+\veps_p)<\vphi(i^*+1+\veps_p)$].}
\end{split}
\]
Hence, 
\begin{equation}\label{esponenteeta}
-\CS(i^*+1+\veps_p)+\CS(1)+\CS(i^*+\veps_p)+\eta<0
\end{equation}
for any $\eta>0$ small enough. Thus, \eqref{eqeta} gives 
\[
\sigma_t(i^*+1+\veps_p)\leq C_{i^*+1+\veps_p} e^{-t \mathfrak{K}}
\]
with
\[
- \mathfrak{K}=\max\{\CS(i^*+1+\veps_p), \CS(1)+\CS(i^*+\veps_p)+\eta\}.
\]
By \eqref{esponenteeta} we get $- \mathfrak{K}=\CS(i^*+1+\veps_p)$ which entails \eqref{induzionesigma}-\eqref{Kp} for $i=i^*+1$ since, as already observed, $\vphi(1)<\vphi(i^*+1+\veps_p)$. The proof can be now concluded by induction for $j=i^*+2,\dots, k_p$ as in Step 1.
\end{proof}

\begin{proof}[Proof of Theorem \ref{alfa1}]
The proof follows the same argument of the one used in the proof of Theorem \ref{alfa<1}. In particular we prove by mathematical induction that
\begin{equation}\label{induzionesigma1}
\sigma_t(i+\veps_p)\leq C_{i+\veps_p}e^{-tK(i+\veps_p)}\quad\text{for every $i=2,\dots,k_p$}
\end{equation}
for suitable constants $0<C_{i+\veps_p}<+\infty$, with
\begin{equation}
\label{K2}
-K(i+\veps_p):=\max\{\CS(i+\veps_p),\vphi(2)(i+\veps_p)\}.
\end{equation}
Since $p\geq2$, we use \eqref{sigma>2} as the fundamental tool for the induction. From the proofs of Theorem \ref{PropW-2} (when $\a\neq1$) and Theorem \ref{PropW-3} (when $\a=1$) one has 
\begin{equation}\label{sigma12}
\sigma_t(2)\leq C^2_2 e^{t\CS(2)}.
\end{equation}
By Lyapunov's and Jensen's inequalities one gets 
\[
\sigma_t(1)\leq\sigma_t(2)^\frac{1}{2}\quad\text{and}\quad\sigma_t(1+\veps)\leq\sigma_t(2)^\frac{1+\veps}{2}
\]
for every $0<\veps<1$.
Combining the above inequalities with \eqref{sigma12} one has
\begin{equation}\label{sigma13}
\begin{split}
&\sigma_t(1)\leq \sqrt{2}d_2(\bar{\mu}_0,\mu_\infty)e^{t\vphi(2)}\quad\text{and}\\
&\sigma_t(1+\veps)\leq \sqrt{2}^{1+\veps}d^{1+\veps}_2(\bar{\mu}_0,\mu_\infty)e^{t\frac{\CS(2)}{2}(1+\veps)}.
\end{split}
\end{equation} 
Using \eqref{sigma13} in \eqref{sigma>2}, it follows that
\begin{equation}\label{sigma14}
\begin{split}
\sigma_t(2+\veps_p)\leq &\Delta_1(2+\veps_p)e^{t\CS(2+\veps_p)}\\
&+e^{t\CS(2+\veps_p)}B_{2+\veps_p}\int_0^t C_1 e^{\tau[-\CS(2+\veps_p)+\vphi(2)(2+\veps_p)]}d\tau.
\end{split}
\end{equation}
Noticing that the first step of induction is $i=2$, one can follow the same steps of the proof of Theorem \ref{alfa<1} using \eqref{sigma13} in place of \eqref{sigma1}, \eqref{sigma14} in place of \eqref{induzione1} and $\vphi(2)=\frac{\CS(2)}{2}$ in place of $\vphi(1)=\CS(1)$.

\end{proof}

\section{Proofs of Proposition \ref{propratechi} and Theorems \ref{d1a2}, \ref{alfa2}}\label{sec:proof2}
We start by proving Proposition \ref{propratechi} which provides rates of convergence in Fourier metrics of suitable orders for any $\a\in(0,2]$.

\begin{proof}[Proof of Proposition \ref{propratechi}]
By convexity of the Fourier distance we know that
\[
\chi_{\alpha+\delta}(\mu_t,\mu_\infty)\leq \sum_{n\geq0}e^{-t}(1-e^{-t})^n\chi_{\alpha+\delta}(\mu_n,\mu_\infty).
\]
So we need a bound for  $\chi_{\a+\delta}(\mu_n,\mu_\infty)$.
By \eqref{stationaryV} one gets
\begin{equation}\label{stationaryhat}
 \hat \mu_\infty(\xi)=\E[\prod_{j=1}^{n}\hat{\mu}_\infty(\beta_{j,n}\xi)].
\end{equation}
Now recall that for every $n\geq1$ if $z_1,\dots,z_n,w_1,\dots\,w_n$ are complex numbers such that $|z_i|<1$ and $|w_i|<1$ for every $i=1,\dots,n$, then
\[
\left|\prod_{i=1}^{n}z_i-\prod_{i=1}^{n}w_i\right|\leq \sum_{i=1}^{n}\left|z_i-w_i\right|.
\]
Using this inequality and \eqref{stationaryhat} one obtains
\begin{eqnarray*}
\chi_{\alpha+\delta}(\mu_n,\mu_\infty)&=&\sup_{\xi\neq0}\frac{|\hat{\mu}_n(\xi)-\hat{\mu}_\infty(\xi)|}{|\xi|^{\alpha+\delta}}\\
&=&\sup_{\xi\neq0}
\E\left(\frac{|\prod_{j=1}^{n}\hat{\bar{\mu}}_0(\beta_{j,n}\xi)-\prod_{j=1}^{n}\hat{\mu}_\infty(\beta_{j,n}\xi)|}{|\xi|^{\alpha+\delta}}\right)\\
&\leq& \sup_{\xi\neq0}\E\left(\sum_{j=1}^n\frac{|\hat{\bar{\mu}}_0(\beta_{j,n}\xi)-\hat{\mu}_\infty(\beta_{j,n}\xi)|}{|\xi|^{\alpha+\delta}}\right)\\
&\leq& \E\left(\sup_{\xi\neq0}\sum_{j=1}^{n}\frac{|\hat{\bar{\mu}}_0(\beta_{j,n}\xi)-\hat{\mu}_\infty(\beta_{j,n}\xi)|}{|\beta_{j,n}\xi|^{\alpha+\delta}}|\beta_{j,n}|^{\alpha+\delta}\right)\\
&=&\sup_{y\neq0}\frac{|\hat{\bar{\mu}}_0(y)-\hat{\mu}_\infty(y)|}{|y|^{\alpha+\delta}}\E\Big(\sum_{j=1}^n\beta^{\alpha+\delta}_{j,n}\Big).\\
\end{eqnarray*}
So, by \eqref{momentiM}, one can write
\[
\chi_{\alpha+\delta}(\mu_t,\mu_\infty)\leq \sum_{n\geq0}e^{-t}(1-e^{-t})^n \chi_{\alpha+\delta}(\bar{\mu}_0,\mu_\infty) \frac{\Gamma(n+\CS(\a+\delta))}{\Gamma(n)\Gamma(\CS(\a+\delta)+1)}
\]
and therefore, using \eqref{serie-gamma},
\[
\chi_{\alpha+\delta}(\mu_t,\mu_\infty)\leq \chi_{\alpha+\delta}(\bar{\mu}_0,\mu_\infty)e^{tS(\alpha+\delta)}.
\]
\end{proof}

In order to prove Theorem \ref{d1a2} we need the following

\begin{proposition}\label{rate1alfa2}
For every two probability measures $\mu_1$, $\mu_2$ on $\R$ such that \\ $\int_\R x^2\mu_1(dx)<+\infty$, $\int_\R x^2\mu_2(dx)<+\infty$ and  $\chi_{2+\delta}(\mu_1,\mu_2) <+\infty$, then
\[
d_1(\mu_1,\mu_2)\leq C\chi^\frac{1}{3(2+\delta)}_{2+\delta}(\mu_1,\mu_2)
\]
with
\[
C:=\left(2^\frac{2}{3}+2^\frac{-1}{3}\right)M^\frac{1}{3}_2\frac{1}{\pi}\left(\frac{2^\frac{3+2\delta}{2+\delta}}{3+2\delta}+\frac{4}{2^\frac{1}{2+\delta}}\right)
\]
and $M_2:=\max\{\int_\R x^2\mu_1(dx), \int_\R x^2\mu_2(dx)\}$.
\end{proposition}
The proof of this proposition can be done following the same argument, with slight changes, of the proof of Theorem 2.21 of \cite{CarrilloToscani}.

\begin{proof}[Proof of Theorem \ref{d1a2}]

It is worth noticing that $\chi_{2+\delta}(\bar{\mu}_0,\mu_\infty)$ is finite. Indeed, as already observed, both $\bar \mu_0$ and $\mu_\infty$ have equal mean (more precisely, zero mean) and equal variance. Thus, Proposition 2.6 in \cite{CarrilloToscani} entails the finiteness of $\chi_{2+\delta}(\bar{\mu}_0,\mu_\infty)$ provided that $\int_\R |x|^{2+\delta}\bar{\mu}_0(dx)<+\infty$ and $\int_\R |x|^{2+\delta}\mu_\infty(dx)<+\infty$; the former integral is finite by hypothesis, the latter is finite since $\CS(2+\delta)<0$.
Under the assumptions of Theorem \ref{alfa2}, one has $\int_\R x^2\mu_t(dx)=\int_\R x^2\mu_\infty(dx)$ for every $t\geq0$ and hence 
one can apply Proposition \ref{rate1alfa2} to get
\[
d_1(\mu_t,\mu_\infty)\leq C\chi^\frac{1}{3(2+\delta)}_{2+\delta}(\mu_t,\mu_\infty)
\]
with $C$ that does not depend on $t$. 
Now Proposition \ref{propratechi} gives
\begin{equation*}
d_1(\mu_t,\mu_\infty)\leq C \chi^\frac{1}{3(2+\delta)}_{2+\delta}(\bar{\mu}_0,\mu_\infty) e^{t\frac{\CS(2+\delta)}{3(2+\delta)}}
\end{equation*}
which proves Theorem \ref{d1a2}.
\end{proof}

 We are now ready to prove Theorem \ref{alfa2}.

\begin{proof}[Proof of Theorem \ref{alfa2}]
Define the integer $k_p\geq2$ and the real number $\veps_p\in(0,1]$ such that $p=k_p+\veps_p$. 
We prove by induction that 
\begin{equation}\label{sigmaLim2}
\sigma_t(j+\veps_p)\leq \left\{
												\begin{array}{ll}
												C_{j+\veps_p}e^{-tR_{j+\veps_p}} & \text{if $\CS(j+\veps_p)\neq\frac{1}{3}\vphi(2+\veps_p)$}\\
												C_{j+\veps_p}te^{-tR_{j+\veps_p}} & \text{if $\CS(j+\veps_p)=\frac{1}{3}\vphi(2+\veps_p)$}
												\end{array}
												\right. 
\end{equation}
for $j=2,\dots,k_p$, where $-R_{j+\veps_p}:=\max\{\CS(j+\veps_p),\frac{1}{3}\vphi(2+\veps_p)\}$.
If $k_p=2$ then $p=2+\veps_p$ and, in order to use \eqref{sigma>2} with $q=2+\veps_p$, we have to compute $\sigma_t(1)$ and $\sigma_t(1+\veps_p)$. Since $\int_\R|x|^{2+\veps_p}\bar\mu_0(dx)<+\infty$ by hypothesis, one clearly has $\int_\R|x|^{2+\veps_p}\mu_n(dx)<+\infty$ for every $n\geq1$. Moreover, $\int_\R|x|^2 \mu_n(dx)=\int_\R|x|^2\mu_\infty(dx)$. 
Then, Proposition \ref{rate1alfa2} and Jensen's inequality give
\[
\begin{split}
\sigma_t(1)&\leq  C \sum_{n\geq0}e^{-t}(1-e^{-t})^n\chi_{2+\veps_p}(\mu_n,\mu_\infty)^\frac{1}{3(2+\veps_p)}\\
& \leq C \Big(\sum_{n\geq0}e^{-t}(1-e^{-t})^n\chi_{2+\veps_p}(\mu_n,\mu_\infty)\Big)^\frac{1}{3(2+\veps_p)}
\end{split}
\]
and hence, arguing as in the proof of Proposition \ref{propratechi},
\[
\sigma_t(1)\leq C_1 e^{t\frac{\CS(2+\veps_p)}{3(2+\veps_p)}}=C_1 e^{\frac{t}{3}\vphi(2+\veps_p)}
\]
where $C_1:=C\chi_{2+\veps_p}(\bar{\mu}_0,\mu_\infty)^\frac{1}{3(2+\veps_p)}$.
Moreover, by \eqref{wass12}, 
\[
\sigma_t(1+\veps_p)=\sum_{n\geq0}e^{-t}(1-e^{-t})^n d^{1+\veps_p}_{1+\veps_p}(\mu_n,\mu_\infty)\leq \Gamma_2^{1+\veps_p}.
\]
Thus, using \eqref{sigma>2} and the above estimates for $\sigma_t(1)$ and $\sigma_t(1+\veps_p)$, one has
\[
\begin{split}
\sigma_t(2+\veps_p)& \leq \Delta_1(2+\veps_p)e^{t\CS(2+\veps_p)}\\
&+C_1 \Gamma_2^{1+\veps_p} B_{2+\veps_p}e^{t\CS(2+\veps_p)}\int_0^t e^{-\tau\CS(2+\veps_p)}e^{\tau\frac{1}{3}\vphi(2+\veps_p)}dt\\
& = \sigma_t(2+\veps_p) \leq \Delta_1(2+\veps_p)e^{t\CS(2+\veps_p)}\\
&+C_1 \Gamma_2^{1+\veps_p} B_{2+\veps_p}e^{t\CS(2+\veps_p)}\int_0^t e^{-\tau\CS(2+\veps_p)\frac{5+3\veps_p}{6+3\veps_p}}dt.
\end{split}
\]
Clearly $\CS(2+\veps_p)\frac{5+3\veps_p}{6+3\veps_p}\neq0$ and hence we get
\[
\sigma_t(2+\veps_p)\leq C_{2+\veps_p}e^{t \max\{\CS(2+\veps_p),\frac{1}{3}\vphi(2+\veps_p)\}}
\]
which is the thesis since $\max\{\CS(2+\veps_p),\frac{\vphi(2+\veps_p)}{3}\}=\frac{\vphi(2+\veps_p)}{3} =-R_{2+\veps_p}$. This concludes the proof if $k_p=2$. On the other hand, if $k_p\geq3$, assume that \eqref{sigmaLim2} holds true for $j=2,\dots,k_p-1$. Using \eqref{sigma>2} we get
\[
\begin{split}
\sigma_t(k_p+\veps_p) & \leq \Delta_1(k_p+\veps_p)e^{t\CS(k_p+\veps_p)}\\
& +B_{k_p+\veps_p}C_1 C_{k_p-1+\veps_p}e^{t\CS(k_p+\veps_p)}\int_0^t e^{\tau[-\CS(k_p+\veps_p)+\frac{1}{3}\vphi(2+\veps_p)]}\sigma_t(k_p-1+\veps_p)d\tau. 
\end{split}
\]
By the inductive hypothesis, $\sigma_t(k_p-1+\veps_p)\leq D$ for a suitable constant $D>0$ and hence
\[
 \begin{split}
\sigma_t(k_p+\veps_p) & \leq \Delta_1(k_p+\veps_p)e^{t\CS(k_p+\veps_p)}\\
& + D B_{k_p+\veps_p}C_1 C_{k_p-1+\veps_p}e^{t\CS(k_p+\veps_p)}\int_0^t e^{\tau[-\CS(k_p+\veps_p)+\frac{1}{3}\vphi(2+\veps_p)]}d\tau. 
\end{split}
\]
and the thesis follows.
\end{proof}

\section{Proof of Theorem \ref{teoremaDistFinita}}\label{Sec_proofUltima}

The proof of this theorem is inspired by the proof of Lemma 3.19 in \cite{ChristophWolf}.
See also Lemma 3.1  in \cite{Stout79}.

\textbf{Proof of Part (i)}
Let $\delta\in[0,1)$ such that $k+\delta=1+\frac{p-\alpha}{p\alpha}$. Let $\cm_0,\cp_0,\tcm_1,\tcp_1$, $\dots,\tcm_{k-1},\tcp_{k-1}$ be given in Proposition \ref{sviluppoFinf}.
Note that, since $|\beta| \not =1$, $c_0^+>0$ and $c_0^->0$>
Recall that if $U$ is a random variable uniformly distributed on $(0,1)$ then $(F^{-1}_0(U),F^{-1}_\infty(U))$ is a coupling for $d_p(\bar{\mu}_0,\mu_\infty)$ and hence
\begin{equation}\label{dist-quantili}
\begin{split}
d_p(\bar{\mu}_0,\mu_\infty)&\leq \Big(\E\Big|F^{-1}_0(U)-F^{-1}_\infty(U)\Big|^p\Big)^{\frac{1}{p}\wedge 1 }\\ &\leq \Big(\E\Big|F^{-1}_0(U)-G^{-1}(U)\Big|^p
\Big)^{\frac{1}{p}\wedge 1 }+\Big(\E\Big|G^{-1}(U)-F^{-1}_\infty(U)\Big|^p\Big)^{\frac{1}{p}\wedge 1 }
\end{split}
\end{equation}
where $G$ is a real-valued function of real argument defined by
\begin{equation}\label{defG}
G(x):=\left\{ \begin{array}{ll}
                \sum_{i=0}^{k-1}\frac{\tcm_i}{|x|^{(i+1)\alpha}} & \text{if $x<-M_1$}\\
                \sum_{i=0}^{k-1}\frac{\tcm_i}{M^{(i+1)\alpha}_1} & \text{if $-M_1\leq x<M_2$}\\
                1-\sum_{i=0}^{k-1}\frac{\tcp_i}{x^{(i+1)\alpha}} & \text{if $x\geq M_2$}\\
                \end{array}
        \right.
\end{equation}
with $M_1>0$, $M_2>0$ being such that $G$ is a distribution function, i.e.
\begin{itemize}
\item [(1)] $\sum_{i=0}^{k-1}\frac{\tcm_i}{M^{(i+1)\alpha}_1}<1$;
\item [(2)] $\sum_{i=0}^{k-1}\frac{\tcm_i}{M^{(i+1)\alpha}_1}\leq 1-\sum_{i=0}^{k-1}\frac{\cp_i}{M^{(i+1)\alpha}_2}$;
\item [(3)] $G'(x)\geq0$ for every  $x\in \R\setminus\{-M_1, M_2\}$.
\end{itemize}
As for (1) and (2), it suffices to choose $M_1$ and $M_2$ sufficiently large. Regarding (3), note that
\[
G'(x)=\left\{ \begin{array}{ll}
                \sum_{i=0}^{k-1}\frac{\tcm_i(i+1)\alpha}{|x|^{(i+1)\alpha+1}} & \text{if $x\in(-\infty,-M_1)$}\\
                0 & \text{if $x\in(-M_1,M_2)$}\\
                \sum_{i=0}^{k-1}\frac{\tcp_i(i+1)\alpha}{x^{(i+1)\alpha+1}} & \text{if $x\in(M_2,+\infty)$},
                \end{array}
                \right.
\]
which is positive for sufficiently large $M_i$'s.
Thus, the function $G$ is a distribution function and $G^{-1}$ is its quantile. With a simple change of variables in \eqref{dist-quantili} we can write

\begin{equation*}
\begin{split}
d_p(\bar{\mu}_0,\mu_\infty)&\leq \Big(\int_\R\Big|G^{-1}(F_0(y))-y\Big|^p dF_0(y)\Big)^{\frac{1}{p} \wedge 1}\!\!\!+\Big(\int_\R\Big|G^{-1}(F_\infty(y))-y\Big|^p dF_\infty(y)\Big)^{\frac{1}{p} \wedge 1}\\
&=C_0+\Big(\int_{(-\Mbar,+\Mbar)^c}\Big|G^{-1}(F_0(y))-y\Big|^p dF_0(y)\Big)^{\frac{1}{p} \wedge 1}+C_\infty\\
&+\Big(\int_{(-\Mbar,+\Mbar)^c}\Big|G^{-1}(F_\infty(y))-y\Big|^p dF_\infty(y)\Big)^{\frac{1}{p} \wedge 1}
\end{split}
\end{equation*}
where $\Mbar \geq \max\{M_1,M_2\}$, $C_0:=\Big(\int_{(-\Mbar,+\Mbar)}\Big|G^{-1}(F_0(y))-y\Big|^p dF_0(y)\Big)^{\frac{1}{p} \wedge 1}<+\infty$,
$C_\infty:=\Big(\int_{(-\Mbar,+\Mbar)}\Big|G^{-1}(F_\infty(y))-y\Big|^p dF_\infty(y)\Big)^{\frac{1}{p} \wedge 1}<+\infty$. Hence,
the finiteness of the $d_p$ distance between $\bar{\mu}_0$ and $\mu_\infty$ follows if we show that
\[
\begin{aligned}
& \int_{(-\Mbar,+\Mbar)^c}\Big|G^{-1}(F_0(y))-y\Big|^p dF_0(y)<+\infty,\\
& \int_{(-\Mbar,+\Mbar)^c}\Big|G^{-1}(F_\infty(y))-y\Big|^p dF_\infty(y)<+\infty.
\end{aligned}
\]
Let us start studying the first integral confining ourselves to the calculus on the interval $(\Mbar,+\infty)$ (the integral on $(-\infty,-\Mbar)$
can be treated in the same way): we introduce the function
\[
H(x):= F_0(x)-G(x)\qquad(x\geq \Mbar).
\]
By hypotheses \eqref{condDatoInizDx} and \eqref{zeta} we deduce
that $H(x)=O\Big(\frac{1}{|x|^{(k+\delta)\alpha}}\Big)$. Assuming, without loss of generality that $F_0(\overline M) > G( M_2)$,
using Taylor expansion of $G^{-1}(F_0(x))$ around $G(x)$ we have that
\begin{eqnarray*}
G^{-1}(F_0(x))&=&G^{-1}(G(x)+H(x))=G^{-1}(G(x))+H(x)\frac{d}{du}G^{-1}(u)\Big|_{u=G(x)}\\
&+&\frac{H^2(x)}{2}\frac{d^2}{du^2}G^{-1}(u)\Big|_{u=G(x)+\theta H(x)}
\end{eqnarray*}
for some $\theta\in(0,1)$. Now, putting $R_x:=G^{-1}(G(x)+\theta H(x))$, we obtain
\begin{equation}\label{stima-integranda}
G^{-1}(F_0(x))-x=\frac{H(x)}{G'(x)}\Big(1-\frac{H(x)}{2}\frac{G''(R_x)G'(x)}{(G'(R_x))^3}\Big).
\end{equation}
From the definition of $G$ given in ($\ref{defG}$) we compute
\[
G'(x)=\sum_{i=0}^{k-1}\frac{\tcp_i (i+1)\alpha}{|x|^{(i+1)\alpha+1}}=\frac{\alpha \cp_0}{|x|^{\alpha+1}}(1+o(1)) \qquad\text{for $x\rightarrow+\infty$}
\]
\[
G''(x)=-\sum_{i=0}^{k-1}\frac{\tcp_i(i+1)\alpha[(i+1)\alpha+1]}{|x|^{(i+1)\alpha+2}}=-\frac{\alpha(\alpha+1) \cp_0}{|x|^{\alpha+2}}(1+o(1))\qquad\text{for $x\rightarrow+\infty$}
\]
and therefore \eqref{stima-integranda} becomes
\begin{equation}\label{stima-integranda2}
\begin{split}
\Big|G^{-1}(F_0(x))-x\Big|=\frac{1}{\alpha \cp_0} &|H(x)||x|^{\alpha+1}(1+o(1)) \\
& \cdot \left|1+\frac{\alpha+1}{\alpha \cp_0}\frac{H(x)}{2}\frac{|R_x|^{2\alpha+1}}{|x|^{\alpha+1}}(1+o(1))\right| \\
\end{split}
\end{equation}
for $x\rightarrow+\infty$.
We now show that $H(x)\frac{|R_x|^{2\alpha+1}}{|x|^{\alpha+1}}=o(1)$ for $x\rightarrow+\infty$. By further increasing $\Mbar$, if needed, we can say that there exist $A_+>0$, $A_->0$ such that for every $x\geq\Mbar$
\[
G_-(x)\leq G(x)\leq G_+(x)
\]
where $G_\pm(x):=1-\frac{A_\pm}{|x|^\alpha}$. In particular
\[
\Big ( \frac{A_+}{1-y} \Big)^{\frac{1}{\alpha}} =G^{-1}_+(y)\leq G^{-1}(y)\leq G^{-1}_-(y)=\Big ( \frac{A_-}{1-y} \Big)^{\frac{1}{\alpha}}
\]
for every $y$ sufficiently close to $1$. Since $G(x)+\theta H(x)\longrightarrow 1$ for $x\rightarrow+\infty$, we obtain
\begin{eqnarray*}
R_x&=& G^{-1}(G(x)+\theta H(x))\leq G^{-1}_-(G(x)+\theta H(x))=\frac{A^\frac{1}{\alpha}_-}{\Big(1-[G(x)+\theta H(x)]\Big)^\frac{1}{\alpha}}\\
&\leq& \frac{A^\frac{1}{\alpha}_-}{\Big(1-G_+(x)-\theta H(x)\Big)^\frac{1}{\alpha}}=\frac{A^\frac{1}{\alpha}_-}{\Big(1-\Big(1-\frac{A_+}{|x|^\alpha}\Big)-\theta H(x)\Big)^\frac{1}{\alpha}}\\
&=& \frac{A^\frac{1}{\alpha}_- |x|}{\Big(A_+-\theta |x|^\alpha H(x)\Big)^\frac{1}{\alpha}}=|x| \Big(\Big(\frac{A_-}{A_+}\Big)^\frac{1}{\alpha}+o(1)\Big)
\end{eqnarray*}
where $o(1)$ is for $x\rightarrow+\infty$. Recalling that $H(x)=O\Big(\frac{1}{|x|^{(k+\delta)\alpha}}\Big)$, we can conclude that
\[
|H(x)|\frac{|R_x|^{2\alpha+1}}{|x|^{\alpha+1}} \leq |H(x)|\frac{|x|^{2\alpha+1}}{|x|^{\alpha+1}}\Big(\Big(\frac{A_-}{A_+}\Big)^\frac{1}{\alpha}+o(1)\Big)^{2\a+1}=o(1)
\]
for $x\rightarrow+\infty$ and therefore, from \eqref{stima-integranda2},
\[
\Big|G^{-1}(F_0(x))-x\Big|=\frac{1}{\alpha \cp_0}|H(x)||x|^{\alpha+1}(1+o(1)).
\]
Hence, for suitable positive constants $C, C', C'', C'''$, we can write
\begin{eqnarray*}
\int_{\Mbar}^{+\infty}\Big|G^{-1}(F_0(x))-x\Big|^p dF_0(x)&\leq& C\int_{\Mbar}^{+\infty}\Big(|H(x)||x|^{\alpha+1}\Big)^p dF_0(x)\\
&\leq& C' \int_{\Mbar}^{+\infty}\Big(\frac{\zeta(|x|)}{|x|^{(k+\delta)\alpha}}|x|^{\alpha+1}\Big)^p dF_0(x)\\
&=& C'\int_{\Mbar}^{+\infty}\zeta^p(|x|)|x|^{p\alpha(1-k-\delta)+p} dF_0(x)\\
&=& C' \int_{\Mbar}^{+\infty}\zeta^p(|x|)|x|^\alpha dF_0(x)
\end{eqnarray*}
From Lemma \ref{integralezeta} below the last term is finite by further increasing $\Mbar$ in order to have $\Mbar>B$.\newline
This argument, which proves that $\int_{\Mbar}^{+\infty}\Big|G^{-1}(F_0(y))-y\Big|^p dF_0(y)<+\infty$ can be extended to the same integral with $(-\infty,-\Mbar)$ as domain of integration.\newline
The integral $\int_{(-\Mbar,+\Mbar)^c}\Big|G^{-1}(F_\infty(y))-y\Big|^p dF_\infty(y)$ can be treated in the
same way noticing that, in view of Proposition $\ref{sviluppoFinf}$, $F_\infty$ satisfies conditions similar to \eqref{condDatoInizSx} and \eqref{condDatoInizDx}
with $\zeta(x)={|x|^{-s+(1+\frac{p-\a}{\a p})\a}}$. This shows that $d_p(\bar{\mu}_0,\mu_\infty)<+\infty$.\newline

\textbf{Proof of Part (ii)} Suppose that $\beta=-1$ (the case $\beta=1$ can be done in an analogous way).
We start as in the proof of Part (i) writing
\[
\begin{split}
&\left(\int_0^1 \Big| F^{-1}_0(u) -F^{-1}_\infty(u)\Big|^p du\right)^{\frac{1}{p}\wedge 1 } \\
&\leq \left(\int_0^{F_0(0)} \Big| F^{-1}_0(u)-F^{-1}_\infty(u)\Big|^p du\right)^{\frac{1}{p}\wedge 1 }
+\left(\int_{F_0(0)}^1 \Big| F^{-1}_0(u)\Big|^p du\right)^{\frac{1}{p}\wedge 1 }\\
&\qquad  \qquad +\left(\int_{F_0(0)}^1 \Big| F^{-1}_\infty(u)\Big|^p du\right)^{\frac{1}{p}\wedge 1 }\\
&\leq \left(\int_0^{F_0(0)} \Big| F^{-1}_0(u)-F^{-1}_\infty(u)\Big|^p du\right)^{\frac{1}{p}\wedge 1 }+\left(\int_{[F^{-1}_0(F_0(0)),+\infty)} |x|^p dF_0(x)\right)^{\frac{1}{p}\wedge 1 }\\
&\qquad \qquad +  \left(\int_{[F_\infty^{-1}(F_0(0)),+\infty)} |x|^p dF_\infty(x)\right)^{\frac{1}{p}\wedge 1 }.\\
\end{split}
\]
The first integral can be treated with the same argument of Part (i); the second integral is finite by hypothesis;
the third, by partial integration, is finite whenever $\int_0^{+\infty}(1-F_\infty(x))x^{p-1}dx$ is finite.
Now, since $\CS(s)<0$, Proposition \ref{sviluppoFinf} gives $1-F_\infty(x)=O(x^{-s})$ and hence $\int_0^{+\infty}(1-F_\infty(x))x^{p-1}dx \leq C  \int_1^{+\infty} x^{p-1-s}dx<+\infty$.
\begin{flushright} $\square$ \end{flushright}

The proof of the following lemma is left to the reader.

\begin{lemma}\label{integralezeta}
Let $\zeta$ be the function defined in {\rm Theorem \ref{teoremaDistFinita}} and suppose that \eqref{zeta} holds true. Then
\[
\int_B^{+\infty}\zeta^p(x)x^\alpha dF_0(x)<+\infty.
\]
\end{lemma}


\appendix
\section{Proof of Proposition \ref{sviluppoFinf}}

In this appendix we prove Proposition \ref{sviluppoFinf}. The main point is to recall the well-known asymptotic expansion for the probability distribution function of an $\a$-stable law with $\a\neq1$:

\begin{proposition}[\cite{ChristophWolf,Ibragimov,Zolotarev1986}]\label{sviluppoFalfaDiv1}
Let $F_\alpha$ be the distribution function of an $\alpha$-stable law of parameters $(\lm,\beta)$ with $\alpha\neq1$.\newline
If $|\beta|\neq1$, then for every $k\geq1$
\begin{eqnarray}
\label{codasx}
&&F_\alpha(x)=\frac{\cm_0}{|x|^{\alpha}}+\frac{\cm_1}{|x|^{2\alpha}}+\dots+\frac{\cm_{k-1}}{|x|^{k\alpha}}+O\Big(\frac{1}{|x|^{(k+1)\alpha}}\Big)\qquad\text{for $x\rightarrow-\infty$}\\
\label{codadx}
&&1-F_\alpha(x)=\frac{\cp_0}{x^{\alpha}}+\frac{\cp_1}{x^{2\alpha}}+\dots+\frac{\cp_{k-1}}{x^{k\alpha}}+O\Big(\frac{1}{x^{(k+1)\alpha}}\Big)\qquad\text{for $x\rightarrow+\infty$}
\end{eqnarray}
where $c^\pm_0$ are related to $(\lambda,\beta)$ by \eqref{constant} and
\begin{equation}\label{Ci}
c^\pm_i:=\frac{(-1)^i \tilde{\lm}^{i+1}\Gamma(\alpha(i+1))\sin(\frac{\pi}{2}(i+1)(\alpha\pm\tilde{\beta}))}{\pi(i+1)!}
\end{equation}
for $i=1,\dots,k-1$ where $(\tilde{\lm},\tilde{\beta})$ are related to $(\lm, \beta)$ by 
\[
\tilde{\beta}:= \frac{2}{\pi}\arctan(\beta\tan(K(\a)\frac{\pi}{2})), \qquad\tilde{\lm}:=\frac{\lm}{\cos(\tilde{\beta}\frac{\pi}{2})}
\]
with
\[
K(\a):= \left\{
				\begin{array}{ll}
				 \a & \text{if $\a\leq1$}\\
				 \a-2 & \text{if $\a>1$}.
				 \end{array}
				 \right.
\]
Moreover, if $\beta=-1$ {\rm[}$\beta=1$, resp.{\rm]}, then \eqref{codasx} {\rm[}\eqref{codadx}, resp.{\rm]} holds true and $1-F_\alpha(x)=O\Big(\frac{1}{|x|^\eta}\Big)$ {\rm[}$F_\alpha(x)=O\Big(\frac{1}{|x|^\eta}\Big)$, 
resp.{\rm]} for $x\rightarrow+\infty$ {\rm[}for $x\rightarrow-\infty$, resp.{\rm]}
for every $\eta>0$.
\end{proposition}

This proposition follows from Theorem 1.4 of \cite{ChristophWolf} by a simple integration of the density therein. See also Section 2.4 of \cite{Ibragimov} and Section 2.4 of \cite{Zolotarev1986}.\newline
On the other hand, if $\alpha=1$ and $c_0^+=c_0^-$, then $F_1$ is a symmetric Cauchy distribution of scale parameter $\lm=c_0^+\pi$ and a straightforward asymptotic expansion gives
\begin{equation}\label{coda1}
F_1(x)=\sum_{i=0}^{k-1}\frac{(-1)^i \lm^{2i+1}}{\pi(2i+1)|x|^{2i+1}}+O\left(\frac{1}{|x|^{2k+1}}\right)\qquad\text{for $x\to-\infty$}
\end{equation}
for every $k\geq1$. Combining \eqref{Finf-Falfa}, Proposition \ref{sviluppoFalfaDiv1} and \eqref{coda1} we obtain Proposition \ref{sviluppoFinf}.


\section{Proof of Theorem \ref{thm1}}\label{AppendixB}

\begin{proof}[Proof of Theorem \ref{thm1}]
The theorem can be proved in a very similar way of Theorem 1 of \cite{BassLadMatth10}. 
In particular it is based on the following simple result:  
{\it Let $(X_{n})_{n \geq 1}$ be a sequence of iid random variables with common distribution function $F_0$. 
Assume that $(a_{jn})_{j \geq 1, n \geq 1}$ is a sequence of positive weights 
such that 
\[
 \lim_{n \to +\infty} \sum_{j=1}^n a_{jn} =a_\infty \qquad \text{and} \qquad \lim_{n \to +\infty} \max_{1 \leq j \leq n} a_{jn}=0.
\]
If $F_0$ satisfy \eqref{NDA-1sym} with $\a=1$, $c_0>0$ and 
\eqref{gamma} holds, then $\sum_{j=1}^n a_{jn} X_j$ converges in law to a Cauchy random variable of scale parameter $\pi a_\infty c_0$ and position parameter $\a_\infty \gamma_0$.}
To prove this result, according to the classical general central limit theorem for array of independent random variables, it is enough to prove that  
  \begin{align}
    \label{condition1-uno-bis}
    &\lim_{n \to +\infty} \zeta_{n}(x) 
    = \frac{a_\infty c_0}{|x|} \qquad (x\not =0), \\
    \label{clt3condtris}
    &\lim_{\epsilon \to 0^+}\lim_{n \to +\infty} \sigma^{2}_{n}(\epsilon) = 0 , \\
    \label{condition3-bis}
    &\lim_{n \to +\infty} \eta_{n} = \a_\infty \gamma_0 
  \end{align}
  are simultaneously satisfied where
  \begin{align*}
    \zeta_n(x) &:=  
    \I \{x<0 \} \sum_{j=1}^{n} Q_{j,n}(x) + \I \{x>0\} \sum_{j=1}^{n} (1-Q_{j,n}(x)) \qquad (x \in \R), \\
  \sigma_{n}^2(\epsilon) &:=  
    \sum_{j=1}^{n}\Big\{ \int_{(-\epsilon,+\epsilon]} x^2\,dQ_{j,n}(x)-\Big( \int_{(-\epsilon,+\epsilon]} x\,dQ_{j,n}(x)\Big)^2 \Big\} \qquad (\epsilon>0), \\
    \eta_{n} &: =   \sum_{j=1}^{n}\Big\{ 1- Q_{j,n}(1) -Q_{j,n}(-1) + \int_{(-1,1]} x\, dQ_{j,n}(x)\Big \}, \\
    Q_{j,n}(x)&:=F_0\big(a_{j,n}^{-1}x\big) \quad \text{with the convention  $F_0(\cdot/0):=\I_{[0,+\infty)}(\cdot)$}. \\
  \end{align*}
See, e.g., Theorem 30 and Proposition 11 in \cite{fristedgray}. Conditions \eqref{condition1-uno-bis} and \eqref{clt3condtris}
can be proved exactly as the analogous conditions of Lemma 5 in \cite{BassLadMatth10}. As for condition \eqref{condition3-bis} note that
\[
\eta_{n}=\sum_{j=1}^n a_{jn} \int_{(-1/a_{jn},1/a_{jn}]} x dF_0(x)+
\sum_{j=1}^n a_{jn}\left[1-F_0\Big(\frac{1}{a_{jn}}\Big)\frac{1}{a_{jn}}-F\Big(-\frac{1}{a_{jn}}\Big)\frac{1}{a_{jn}}\right].
\]
Using the assumption on $F_0$ and $(a_{jn})_{jn}$ it follows immediately that
\[
\lim_n \sum_{j=1}^n a_{jn} \int_{(-1/a_{jn},1/a_{jn}]} x dF_0(x)= a_\infty \gamma_0
\]
and
\[
 \qquad \lim_n \sum_{j=1}^n a_{jn}\left[1-F_0\Big(\frac{1}{a_{jn}}\Big)\frac{1}{a_{jn}}-F\Big(-\frac{1}{a_{jn}}\Big)\frac{1}{a_{jn}}\right]=a_\infty(c_0-c_0)=0.
\]
This gives  \eqref{condition3-bis}.
Using this result one obtains the analogous of Lemma 5  in \cite{BassLadMatth10} for $\a=1$. At this stage the proof can be completed following the proof of Theorem 1 in 
 \cite{BassLadMatth10}. 
\end{proof}


\begin{thebibliography}{1}


\bibitem{alsmeyerbiggins}
\textsc{Alsmeyer, G.} and \textsc{Meiners, M.} (2011).
Fixed points of the smoothing transform: two-sided solutions. {\em Probab. Theory Relat. Fields}.
DOI: 10.1007/s00440-011-0395-y





\bibitem{BaLa}\textsc{Bassetti, F.} and \textsc{Ladelli, L.} (2010).
Self similar solutions in one-dimensional kinetic models:
a probabilistic view.  To appear in \textsc{Ann.App.Prob.} arXiv:1003.5527.


\bibitem{BassLadMatth10}
\textsc{Bassetti, F., Ladelli, L.} and \textsc{Matthes, D.} (2011).
Central limit theorem for a class of one-dimensional kinetic
  equations. {\em  Probab. Theory Related Fields} \textbf{150}  77-109.



\bibitem{BaLaRe}
 \textsc{Bassetti, F., Ladelli, L.} and \textsc{Regazzini, E.} (2008).
  {Probabilistic study of the speed of approach to equilibrium for an inelastic Kac model.}
  \textit{J. Stat. Phys.} \textbf{133} 683--710.


\bibitem{BaLaTo}
\textsc{Bassetti, F., Ladelli, L.} and \textsc{Toscani, G.} (2011).
Kinetic models with randomly perturbed binary collisions.
{\em J. Stat. Phys.} \textbf{142} 686-709.

\bibitem{BassettiToscani}
\textsc{Bassetti, F.} and \textsc{Toscani, G.} (2010).
Explicit equilibria in a kinetic model of gambling.
{\em Phys. Rev. E} \textbf{81} 066115.


\bibitem{BCC}
\textsc{Basu, B., Chackabarti, B.K., Chackavart, S.R.} and \textsc{Gangopadhyay, K.}
(Eds.) (2010). \textit{Econophysics \& Economics of Games, Social Choices
and Quantitative Techniques}.  Springer Verlag, Milan.



\bibitem{BoCe}
\textsc{Bobylev, A.V.} and \textsc{Cercignani, C.} (2003).
Self-similar asymptotics for the Boltzmann equation
with inelastic and elastic interactions, {\em J. Statist. Phys.}  \textbf{1} 10 333--375.



\bibitem{CeGaBoBis}
\textsc{Bobylev, A.V., Cercignani, C.} and \textsc{Gamba, I.M.} (2008).
Generalized kinetic Maxwell type models of granular gases.
In: {\it Mathematical models of granular matter Series: Lecture Notes in Mathematics}, Vol. 1937, G. Capriz,
P. Giovine, P. M. Mariano (eds.) Berlin-Heidelberg-New York: Springer, 23–58.


\bibitem{CeGaBo}
\textsc{Bobylev, A.V., Cercignani, C.} and \textsc{Gamba, I.M.} (2009).
\newblock On the self-similar asymptotics for generalized nonlinear kinetic
  maxwell models.
\newblock {\em Comm. Math. Phys.} \textbf{291} 599--644.


\bibitem{CarrilloToscani} \textsc{Carrillo, J.A.} and  \textsc{Toscani, G.} (2007). Contractive probability metrics and asymptotic behaviour of dissipative kinetic equations. \textit{Riv. Mat. Univ. Parma} \textbf{6}, 75--198.

\bibitem{Cramer}
\textsc{Cramer, H.} (1963). On asymptotic expansions for sums of independent random
  variables with a limiting stable distribution. \textit{Sankhy\=a Ser. A}
{\bf 25}  13-24. \textit{ Addendum, ibid.}  216.



\bibitem{ChristophWolf} \textsc{Christoph, G.} and \textsc{Wolf, W.} (1992). \textit{Convergence theorems with a stable limit law}. Akademie Verlag.


\bibitem{dolera}
\textsc{Dolera, E., Gabetta, E.} and \textsc{Regazzini, E.} (2009).
Reaching the best possible rate of convergence
to equilibrium for solutions of Kac's equation via central limit theorem.
\textit{Ann. Appl. Probab.} {\bf 19} 186-209.


\bibitem{dolera2}
\textsc{Dolera, E.} and  \textsc{Regazzini, E.} (2010).
The role of the central limit theorem in discovering sharp rates
of convergence to equilibrium for the solution of the Kac equation.
\textit{Ann. Appl. Probab.} {\bf 20} 430-461.



\bibitem{DurrettLiggett1983}
\textsc{Durrett, R.} and \textsc{Liggett, T.M.}  (1983).
\newblock Fixed points of the smoothing transformation.
\newblock {\em Z. Wahrsch. Verw. Gebiete} \textbf{64} 275--301.


\bibitem{FillJans01} \textsc{Fill, J.A.} and \textsc{Janson, S.} (2001). Approximating the limiting Quicksort distribution. \textit{Random Structures and Algorithms} \textbf{19}, 1-29.

\bibitem{fristedgray}
\textsc{Fristedt, B.} and \textsc{Gray, L.} (1997).
\newblock {\em A modern approach to probability theory}.
\newblock Probability and its Applications. Birkh\"auser Boston Inc., Boston,
  MA.

  \bibitem{GabettaRegazziniWM}
  \textsc{Gabetta, E.} and \textsc{Regazzini, E.} (2010).
  {Central limit theorem for the solution of the Kac equation: Speed of approach to equilibrium
in weak metrics.}
   \textit{Probab. Theory  Related Fields}
  {\bf 146}   451-480.

\bibitem{GabettaRegazzini2012}
  \textsc{Gabetta, E.} and \textsc{Regazzini, E.} (2012).
  {Complete characterization of convergence to equilibrium for an inelastic Kac model.}
   \textit{J. Statist. Phys.}
  {\bf 147}   1007-1019.

\bibitem{Ibragimov} \textsc{Ibragimov, I.~A.} and \textsc{Linnik, Y.~V.} (1971).
  \textit{Independent and Stationary Sequences of Random Variables}.
  Wolters-Noordhoff Publishing, Groningen.



\bibitem{Kac}
\textsc{Kac., M.}  {Foundations of kinetic theory} (1956).
  In: \textit{Proceedings of the Third Berkeley Symposium on Mathematical Statistics and Probability,
    1954--1955} \textbf{3}  171--197.
  University of California Press, Berkeley and Los Angeles.


\bibitem{Liu1998}
\textsc{Liu, Q.} (1998).
 Fixed points of a generalized smoothing transformation and
  applications to the branching random walk.
{\em Adv. in Appl. Probab.}  \textbf{30} 85--112.



\bibitem{MatthesToscani}
  \textsc{Matthes, D.} and \textsc{Toscani, G.} (2008).
  On steady distributions of kinetic models of conservative economies.
  \textit{ J. Statist. Phys.} {\bf 130}  1087-1117.

\bibitem{MaTo}
\textsc{Matthes, D.}  and  \textsc{Toscani, G.} (2010).  Propagation of Sobolev regularity for a class of random kinetic models on the real line. \emph{Nonlinearity}
\textbf{23} 2081-2100.

\bibitem{McKean1966}
\textsc{McKean Jr, H.R.} (1966).
  {Speed of approach to equilibrium for {K}ac's caricature of a {M}axwellian gas.}
  \textit{Arch. Rational Mech. Anal.} \textbf{21} 343--367.




\bibitem{NPT}
\textsc{Naldi, G., Pareschi, L.} and \textsc{Toscani, G.} (Eds.) (2010). \textit{Mathematical
Modeling of Collective Behavior in Socio-Economic and Life Sciences}.
Birkhauser, Boston.




\bibitem{PerversiRegazzini}
\textsc{Perversi, E.} and \textsc{Regazzini, E.} (2012).
\newblock Sufficient and necessary conditions for the convergence to equilibrium for the solution of some kinetic equations.
In preparation.



\bibitem{PulvirentiToscani}
\textsc{Pulvirenti, A.} and \textsc{Toscani, G.} (2004).
\newblock Asymptotic properties of the inelastic {K}ac model.
\newblock {\em J. Statist. Phys.}  \textbf{114} 1453--1480.

\bibitem{RachevRuschendorf} \textsc{Rachev, S. T.} and \textsc{Ruschendorf, L.} (1998). \textit{Mass transportation problems}, Vol. 2. Springer, New York.




\bibitem{Stout79} \textsc{Stout, W.} (1979).
Almost sure invariance principles when $E(X_1^2)=+\infty$.
{\it  Z. Wahrsch. Verw. Gebiete}
 {\bf 49} 23--32.

\bibitem{BahrEsseen}
\textsc{von Bahr, B.} and \textsc{Esseen, C.G.} (1965).
  Inequalities for the $r$th absolute moment of a sum
  of random variables, $1\leq r\leq 2$.
  \textit{Ann. Math. Statist.} \textbf{36}  299--303.



		
\bibitem{Wild1951}
\textsc{Wild, E.}  (1951).
  {On {B}oltzmann's equation in the kinetic theory of gases.}
  \textit{Proc. Cambridge Philos. Soc.} \textbf{47} 602--609.


%
\bibitem{Zolotarev1986}
\textsc{Zolotarev, V.M.} (1986). \textit{One-Dimensional Stable
Distributions}.
{Translations of Mathematical Monographs}
{\bf 65} AMS, Providence.

\end{thebibliography}
\end{document}